\documentclass[11pt]{article}
\usepackage[top=1in, bottom=1in, left=1in, right=1in]{geometry}

\usepackage[linktocpage,colorlinks,linkcolor=blue,anchorcolor=blue,citecolor=blue,urlcolor=blue,pagebackref]{hyperref}
\usepackage{euscript,mdframed}
\usepackage{microtype,todonotes,relsize}
\usepackage{amsmath,amsthm}
\usepackage{algpseudocode}

\usepackage{accents}
\usepackage{comment,url,graphicx,relsize}
\usepackage{amssymb,amsfonts,amsmath,amsthm,amscd,dsfont,mathrsfs,mathtools,nicefrac, bm}
\usepackage{float,psfrag,epsfig,color,xcolor,url}
\usepackage{epstopdf,bbm,mathtools,enumitem}
\usepackage{subfigure}
\usepackage[ruled,vlined]{algorithm2e}
\usepackage{tablefootnote}
\graphicspath{{Plots/}}
\usepackage{diagbox}
\usepackage{tikz}
\usetikzlibrary{calc,patterns,angles,quotes}
\usepackage{makecell}
\usepackage{multirow}
\usepackage{booktabs}
\usepackage{cleveref}
\usepackage[normalem]{ulem} 

\def\R{\mathbb{R}} 

\newcommand{\E}{\mathbb{E}}

\providecommand{\argmin}{\mathop\mathrm{arg min}}
\providecommand{\dom}{\mathop\mathrm{dom}}

\newcommand{\prox}{\text{prox}}

\ifdefined\nonewproofenvironments\else
\ifdefined\ispres\else
\renewenvironment{proof}{\noindent\textbf{Proof.}\hspace*{.3em}}{\qed\\}
\newenvironment{proof-sketch}{\noindent\textbf{Proof Sketch}
  \hspace*{0.em}}{\qed\bigskip\\}
\newenvironment{proof-idea}{\noindent\textbf{Proof Idea}
  \hspace*{0.em}}{\qed\bigskip\\}
\newenvironment{proof-of-lemma}[1][{}]{\noindent\textbf{Proof of Lemma {#1}.}
  \hspace*{0.em}}{\qed\\}
\newenvironment{proof-of-corollary}[1][{}]{\noindent\textbf{Proof of Corollary {#1}.}
  \hspace*{0.em}}{\qed\\}
\newenvironment{proof-of-theorem}[1][{}]{\noindent\textbf{Proof of Theorem {#1}.}
  \hspace*{0.em}}{\qed\\}
\newenvironment{proof-attempt}{\noindent\textbf{Proof Attempt}
  \hspace*{0.em}}{\qed\bigskip\\}

\fi
\newtheorem{theorem}{Theorem}[section]
\newtheorem{lemma}{Lemma}[section]

\newtheorem{proposition}{Proposition}[section]
\newtheorem{assumption}{Assumption}[section]
\newtheorem{remark}{Remark}[section]

\newtheorem{definition}{Definition}[section]

\usetikzlibrary{calc}

\allowdisplaybreaks
\usepackage[authoryear,round]{natbib}
\renewcommand*{\backref}[1]{\ifx#1\relax \else Page #1 \fi}
\renewcommand*{\backrefalt}[4]{%
  \ifcase #1 \footnotesize{(Not cited.)}%
  \or        \footnotesize{(Cited on page~#2.)}%
  \else      \footnotesize{(Cited on pages~#2.)}%
  \fi
}

\definecolor{orangePumpkin}{RGB}{211,84,0}
\definecolor{forestgreen}{rgb}{0.13, 0.55, 0.13}

\newcommand*{\colorboxed}{}
\def\colorboxed#1#{%
  \colorboxedAux{#1}%
}
\newcommand*{\colorboxedAux}[3]{%
  \begingroup
    \colorlet{cb@saved}{.}%
    \color#1{#2}%
    \boxed{%
      \color{cb@saved}%
      #3%
    }%
  \endgroup
}
\numberwithin{equation}{section}
\usepackage{xcolor} 
\usepackage{marginnote}

\newcommand{\todol}[2][]{{%
 \let\marginpar\marginnote
 \reversemarginpar
 \renewcommand{\baselinestretch}{0.8}%
 \todo[color=yellow]{#2}}}

\newcommand{\xhat}{\hat{x}}
\newcommand{\xkp}{x^{k+1}}
\newcommand{\xk}{x^k}
\newcommand{\xhk}{\xhat^k}

\newcommand{\xik}{\xi_k}

\newcommand{\rr}{\mathbb{R}}

\renewcommand{\(}{\left(}
\renewcommand{\)}{\right)}
\renewcommand{\[}{\left[}
\renewcommand{\]}{\right]}

\newcommand{\la}{\left\langle}
\newcommand{\ra}{\right\rangle}

\newcommand{\Cbar}{\bar{C}}

\newcommand{\benv}[1]{#1_\lambda^\omega}
\newcommand{\bprox}[1]{\prox_{\lambda,#1}^\omega}

\newcommand{\sfx}{\widetilde{f_x}}
\newcommand{\sfxk}{\widetilde{f}_{\xk}}
\newcommand{\sfxkp}{\widetilde{f}_{\xkp}}

\newcommand{\sFx}{\widetilde{F_x}}
\newcommand{\sFxk}{\widetilde{F}_{\xk}}

\newcommand{\benvF}{\benv{F}}
\newcommand{\bproxF}{\bprox{F}}

\newcommand{\benvvp}{\benv{\varphi}}
\newcommand{\bproxvp}{\bprox{\varphi}}

\newcommand{\dsum}{\sum\limits}
\newcommand{\dsymomelam}{D^{\hbox{\scriptsize sym}}_{\omega, \lambda}}
\newcommand{\dsymphic}[1]{D^{\hbox{\scriptsize sym}}_{\omega, #1}}
\newcommand{\sdome}{\widetilde{D}_\omega}
\newcommand{\some}{\widetilde{\omega}}

\newcommand{\sml}{\widetilde{L}}

\newcommand{\xbar}{\bar{x}}

\newcommand{\smf}[1]{\widetilde{f}_{#1}}

\newcommand{\kbar}{\bar{k}}
\newcommand{\xkbar}{x^{\kbar}}

\newcommand{\snome}{\widetilde{\nabla}\omega}
\newcommand{\Bcal}{\mathcal{B}}

\newcommand{\bcalk}{\Bcal_k}
\newcommand{\xiki}{\xi_{k,i}}

\title{On Nonsmooth and Relatively Weakly Convex Minimization}

\author{
{Chen Jiang} \thanks{Department of Industrial and Systems Engineering, University of Minnesota.  \texttt{jian0649@umn.edu}}
\and
{Jiawen Bi} \thanks{Department of Industrial and Systems Engineering, University of Minnesota.  \texttt{bi000050@umn.edu}}
\and
{Shuzhong Zhang} \thanks{Department of Industrial and Systems Engineering, University of Minnesota.   \texttt{zhangs@umn.edu}}
}
\date{}
\begin{document}

\maketitle
\begin{abstract}
    Composite optimization 
    plays a central role in 
    modern machine learning and signal processing, as it offers 
    a natural balance between data fidelity and structural properties. In this paper, we study composite optimization
    in the setting where both components are nonsmooth and nonconvex. We start with a deterministic Bregman proximal subgradient method that converges under subgradient upper-bound conditions. This approach relaxes the standard requirement on the convexity of the regularization term, thus accommodating a broader range of applications. To extend this to the stochastic regime, we develop a model-based minimization method under a relative Lipschitz condition and establish a convergence rate of $\mathcal{O}(\varepsilon^{-4})$. We also extend the framework with convergence guarantees to the setting where the distance generating function and its gradient are accessible only through a stochastic oracle.
\end{abstract}


\section{Introduction}\label{sec:1}

In this work, we study composite optimization problems of the form
\begin{align}
\label{eq:pro}
\min_{x\in\R^d} F(x):=f(x)+r(x),
\end{align}
where $f$ and $r$ are both relatively weakly convex (to be specified later), possibly nonsmooth. 
Such problems appear in machine learning, signal and image processing, where $f$ usually represents a loss or data fidelity term, while $r$ induces additional structures on the solution, such as sparsity, low-rankedness, robustness, or feasibility constraints \citep{shen2018nonconvex,chen2014convergence,huo2023minimization}. A standard approach to solving such composite problems is 
the so-called proximal gradient method 
when $f$ is differentiable, which updates the iterate as follows: 
\begin{align}
x_{k+1} = \argmin_{x} \left\{\langle \nabla f(x_k),x-x_k\rangle+r(x)
+\frac{1}{2\eta_k}\|x-x_k\|^2\right\}.
\label{eq:prox_grad_formu}
\end{align} Classical convergence theory for this approach has primarily been developed under the assumption that $f$ is smooth and $r$ is convex \citep{beck2009fast,nesterov2013gradient,parikh2014proximal}. However, many real-world models of interest violate these assumptions. In particular, the loss term $f$ may be nonsmooth, so the gradient $\nabla f(x_k)$ in \eqref{eq:prox_grad_formu} is not available, and it must be replaced by a subgradient $g_k\in \partial f(x_k)$ \citep{clarke1990optimization,rockafellar1997convex}. Furthermore, the structural term $r$ can be nonconvex in sparse optimization and statistical learning, where penalties such as $\ell_p$ regularization with $0<p<1$ and SCAD are used to reduce the bias of convex regularization and promote stronger sparsity \citep{wang2010sparse,fan2001variable}. Difference-of-convex (DC) regularization provides another application of nonconvex structural terms \citep{esser2013method,yin2015minimization,huo2023minimization}. 

Furthermore, the Euclidean proximal term in \eqref{eq:prox_grad_formu} may not reflect the geometry of some structured problems. In these settings, Bregman distances offer a naturally extended proximity measure. Let $\omega:\R^d\to\R$ be a continuously differentiable function that is $1$-strongly convex with respect to a norm $\|\cdot\|$. The Bregman distance generated by $\omega$ is defined as
\begin{align*}
D_\omega(y,x)
:=
\omega(y)-\omega(x)-\langle \nabla\omega(x),y-x\rangle.
\end{align*}
Replacing the Euclidean proximal term in \eqref{eq:prox_grad_formu} by $D_\omega$ leads to the Bregman proximal gradient or subgradient update
\begin{align}
x_{k+1}
\in
\argmin_{x\in\R^d}
\left\{
\langle g_k,x-x_k\rangle
+r(x)
+\frac{1}{\eta_k}D_\omega(x,x_k)
\right\},\nonumber
\end{align}
where $g_k=\nabla f(x_k)$ in the differentiable case and $g_k\in\partial f(x_k)$ in the nonsmooth case. Bregman proximal methods have been studied under several structural assumptions, including relative smoothness and smooth adaptability \citep{bauschke2017descent,lu2018relatively,hanzely2021accelerated,bolte2018first}. 

For nonsmooth and nonconvex objectives, we use the Bregman Moreau envelope and the associated Bregman proximal mapping of the full objective $F$. For a function $F:\R^d\to\R\cup\{+\infty\}$ and a parameter $\lambda>0$, these are given by
\begin{align}
F_\lambda^\omega(x):=
\min_{y\in\R^d}
\left\{
F(y)+\frac{1}{\lambda}D_\omega(y,x)
\right\},\qquad 
\label{eq:def_of_me}
\operatorname{prox}^{\omega}_{\lambda F}(x)
:=\argmin_{y\in\R^d}\left\{F(y)+\frac{1}{\lambda}D_\omega(y,x)\right\}.
\end{align}
For a general nonconvex $F$, the proximal subproblem in \eqref{eq:def_of_me} need not have a unique minimizer. Relative weak convexity as defined below, provides a sufficient condition for uniqueness.
\begin{definition}[Relative Weak Convexity \citep{zhang2018convergence}]
\label{def:rwc}
Let $\omega:\R^d\to\R$ be a continuously differentiable and strongly convex distance generating function. A proper lower-semicontinuous function $\varphi:\R^d\to\R$ is said to be $\rho\omega$-relatively weakly convex if the function $\varphi+\rho\omega$ is convex. Equivalently, $\varphi$ is $\rho\omega$-relatively weakly convex if for every $x,z\in\R^d$ and every subgradient $g\in\partial\varphi(z)$, the following inequality holds:
\begin{align}
\varphi(x) \geq \varphi(z)+\langle g,x-z\rangle-\rho D_\omega(x,z).\nonumber
\end{align}
\end{definition}

If $F$ is $\rho\omega$-relatively weakly convex and $\lambda\in(0,1/\rho)$, then the objective in \eqref{eq:def_of_me} is strongly convex in $y$, and so $\operatorname{prox}^{\omega}_{\lambda F}(x)$ is well-defined and is in fact single-valued. For fixed $\lambda>0$, 
denote $\hat x^k:=\operatorname{prox}_{\lambda F}^{\omega}(x^k)$.
Following \cite{zhang2018convergence}, we measure the stationarity by the symmetrized Bregman distance between $x$ and its Bregman proximal point in the following definition.
\begin{definition}[Bregman Stationarity Measure \citep{zhang2018convergence}]
\label{def:bregman_stationarity}
For a given parameter $\lambda \in (0, 1/\rho)$, the stationarity of $F$ at a point $x \in \R^d$ is measured by the symmetrized Bregman distance between $x$ and its Bregman proximal point:
\begin{align}
\dsymomelam(\operatorname{prox}^{\omega}_{\lambda F}(x),x):=
\frac{1}{\lambda^2}\bigl(D_\omega(\operatorname{prox}^{\omega}_{\lambda F}(x),x)
+D_\omega(x,\operatorname{prox}^{\omega}_{\lambda F}(x))\bigr).
\label{eq:def_of_bsm}
\end{align}
\end{definition}
To align with standard stationary measure, we call $x$ an $\varepsilon$-stationary point if $\dsymomelam(\operatorname{prox}^{\omega}_{\lambda F}(x),x)\leq\varepsilon^2.$
This quantity vanishes exactly when $x=\operatorname{prox}^{\omega}_{\lambda F}(x)$, which gives stationarity of $F$ through the optimality condition of \eqref{eq:def_of_me}. Moreover, as discussed in \cite{zhang2018convergence},  this measure dominates the squared proximal residual
\begin{align}
\dsymomelam(\operatorname{prox}^{\omega}_{\lambda F}(x),x)\geq\frac{1}{\lambda^2}
\left\|x-\operatorname{prox}^{\omega}_{\lambda F}(x)\right\|^2.
\nonumber
\end{align}

Existing analyses cover two closely related settings. The
stochastic Bregman model-based framework of
\citet{davis2018stochastic} allows the full composite model
$f_x(\cdot;\xi)+r(\cdot)$ to be relatively weakly convex and
establishes convergence in a one-sided Bregman proximal
distance. On the other hand, \citet{zhang2018convergence}
analyzes the symmetrized Bregman stationarity measure
for stochastic mirror descent with a convex regularizer. Our
contribution is to establish guarantees in the symmetrized
Bregman stationarity measure for general stochastic
model-based updates while retaining a relatively weakly convex
regularizer in the subproblem. We also develop alternative
subgradient and model-growth conditions and extend the analysis
to the setting where the distance generating function itself is
accessed through stochastic oracles.

\subsection{Contributions}\label{subsec:contrib}
This paper develops Bregman proximal methods for composite problems $F=f+r$, where both components may be nonsmooth and nonconvex and are relatively weakly convex with respect to the same distance generating function. Our contributions are as follows.

\begin{itemize}
\item We study a deterministic Bregman proximal subgradient method that retains the relatively weakly convex regularizer $r$ in the proximal subproblem. Under three alternative subgradient growth conditions, we derive a nonasymptotic convergence guarantee for the Bregman stationarity measure and an
$\mathcal O(\varepsilon^{-4})$ iteration complexity to locate an
$\varepsilon$-stationary point.

\item We extend the deterministic analysis to a stochastic model-based method with exact Bregman geometry. The method replaces $f$ by a stochastic model while preserving the composite structure of the subproblem. Under relative weak convexity and a relative Lipschitz condition, we establish an $\mathcal O(\varepsilon^{-4})$ iteration complexity to locate an expected $\varepsilon$-stationary point.

\item We further consider the setting in which the distance generating function and its gradient are also accessed through stochastic oracles. Using minibatch estimates of both the stochastic model and the Bregman distance, we derive an explicit stationarity bound and an $\mathcal O(\varepsilon^{-4})$ iteration complexity to locate an expected $\varepsilon$-stationary point under the additional minibatch condition in \Cref{thm:model-sbmm}.
\end{itemize}

\subsection{Literature Review and Related Work}

\paragraph{Proximal gradient and proximal subgradient methods.}
Proximal gradient methods are classical for composite optimization problems where the objective is written as the sum of a loss term and a structural regularizer. They originate from forward backward splitting methods for monotone operator problems \citep{passty1979ergodic,bruck1977weak} and were later developed for convex optimization \citep{fukushima1981generalized,combettes2005signal}. For convex composite optimization with a smooth loss and convex regularizer, proximal gradient methods achieve the standard $\mathcal O(\varepsilon^{-1})$ complexity in objective accuracy, while accelerated proximal gradient methods achieve the optimal $\mathcal O(\varepsilon^{-1/2})$ complexity \citep{nesterov2013gradient,beck2009fast,parikh2014proximal}. When the loss term is nonsmooth but convex, the gradient can be replaced by a subgradient, leading to proximal subgradient methods \citep{rockafellar1997convex}. For nonsmooth nonconvex functions, generalized subdifferentials provide the corresponding first order objects \citep{clarke1990optimization}. Recent work has further studied proximal subgradient type methods for weakly convex objectives and established convergence guarantees in Moreau envelope based stationarity measures \citep{davis2019stochastic,asi2019importance,deng2021minibatch,zhu2023unified,gao2024stochastic}. 

\paragraph{Bregman proximal methods.}
A natural extension of proximal gradient methods replaces the Euclidean distance in the proximal mapping by a Bregman distance \citep{bregman1967relaxation}. The early Bregman proximal methods were studied in \citet{censor1992proximal,teboulle1992entropic,chen1993convergence,eckstein1993nonlinear}. More recent Bregman proximal gradient methods show that the global Lipschitz gradient assumption on the differentiable term can be replaced by assumptions stated relative to the Bregman distance \citep{bauschke2017descent,bolte2018first}. Several variants have also been proposed, including Bregman proximal gradient methods with extrapolation and inertial terms \citep{zhang2019bregman,mukkamala2020convex}, as well as other methods that further weaken Euclidean smoothness requirements \citep{reem2019telescopic,zhao2022two,zhu2021level}. 

\paragraph{Stochastic proximal and model based methods.}
The stochastic approximation scheme dates back to \citet{robbins1951stochastic}, and stochastic proximal subgradient methods combine stochastic subgradient steps with proximal updates for the regularizer. Classical analyses usually assume either convexity or smoothness \citep{xiao2014proximal,defazio2014saga,lan2018optimal,allen2018katyusha,allen2016variance,ghadimi2013stochastic,lei2017non,reddi2016stochastic}. More recent work establishes stationarity guarantees for weakly convex stochastic objectives and relaxes global Lipschitz assumptions \citep{davis2019stochastic,asi2019importance,deng2021minibatch,zhu2023unified,gao2024stochastic}. A related framework is stochastic model based minimization, where each iteration minimizes a stochastic local model plus a proximal term \citep{davis2019stochastic}. This framework covers stochastic proximal point, stochastic proximal subgradient, and stochastic prox linear methods \citep{asi2019stochastic,davis2019proximally,drusvyatskiy2018error,zhang2022stochastic}. In the Bregman stochastic setting, \citet{davis2018stochastic} allows the full composite model to be relatively weakly convex and establishes a one-sided Bregman proximal bound, whereas \citet{zhang2018convergence} analyzes a symmetrized Bregman stationarity measure for stochastic mirror descent with a convex regularizer. Other recent work has studied differentiable relatively smooth objectives and related stochastic Bregman methods \citep{ding2025nonconvex,fatkhullin2024taming, wang2024bregman}.

\subsection{Notation}

Throughout, $\|\cdot\|$ and $\langle \cdot,\cdot\rangle$ denote the Euclidean norm and inner product on $\R^d$. For a proper function $h:\R^d\to\R\cup\{+\infty\}$, its effective domain is denoted by
$\dom h:=\{x\in\R^d:h(x)<\infty\}$.
A function $\omega:\R^d\to\R$ is $\sigma$-strongly convex if
$\omega-\frac{\sigma}{2}\|\cdot\|^2$ is convex, and is $L$-smooth if $\nabla\omega$ is $L$-Lipschitz continuous. We write $D_\omega(y,x)$ for the Bregman distance generated by $\omega$. For a function $h$, $\partial h(x)$ denotes its subdifferential.  If $h$ is $\rho\omega$-relatively weakly convex, then $\partial h(x)=\partial(h+\rho\omega)(x)-\rho\nabla\omega(x).$ We use $\E[\cdot]$ for expectation and $\E_k[\cdot]$ for conditional expectation with respect to the history up to iteration $k$.


\section{Bregman Proximal Subgradient Method}\label{sec:BPG}

In this section, we study the deterministic composite problem
\begin{align*}
\min_{x\in\R^d}\ F(x):=f(x)+r(x).
\end{align*}
We focus on the case where both $f$ and $r$ may be nonsmooth and nonconvex. This setting is not covered by the standard Bregman proximal subgradient theory, which usually relies on the convexity of the regularizer term $r$.
We consider the Bregman proximal subgradient update
\begin{align}
\label{eq:bpg}
x^{k+1}
\in
\argmin_{x\in\R^d}
\left\{
\langle g_k,x-x^k\rangle
+r(x)
+\frac{1}{t_k}D_\omega(x,x^k)
\right\},
\end{align}
where $g_k\in\partial f(x^k)$. The method uses only a subgradient of $f$, while the possibly nonconvex regularizer $r$ is handled by the Bregman proximal step. The goal of this section is to prove convergence of this method under relative weak convexity assumptions.
Specifically, we assume that $f$ and $r$ are $\rho_f\omega$- and $\rho_r\omega$-relatively weakly convex with respect to the same distance generating function $\omega$.
If the two components are initially relatively weakly convex with respect to different distance generating functions $\omega_1$ and $\omega_2$, respectively, one may take $\omega:=\omega_1+\omega_2$, since both components remain relatively weakly convex with respect to this common distance generating function. Consequently, the full objective $F=f+r$ is $\rho_F\omega$-relatively weakly convex with $\rho_F:=\rho_f+\rho_r$. 
Under this assumption and the subgradient conditions stated below, we prove convergence in the Bregman stationarity measure defined in \eqref{eq:def_of_bsm}. Under additional subgradient growth conditions, we prove that the Bregman proximal subgradient method requires $\mathcal O(\varepsilon^{-4})$ iterations to locate an $\varepsilon$-stationary point.

The rest of this section is organized as follows. We first introduce the assumptions used in the deterministic analysis, then give the formal Bregman proximal subgradient algorithm, and finally state the convergence theorem.

\subsection{Preliminaries and Assumptions}

We state here the assumptions used in the deterministic analysis. The first assumption is the relative weak convexity condition on the two components of the objective.

\begin{assumption}\label{2assum:rwc}
The functions $f:\R^d\to\R$ and
$r:\R^d\to\R$ are proper and lower semicontinuous.
Moreover, $f$ and $r$ are $\rho_f\omega$-RWC and
$\rho_r\omega$-RWC, respectively, where
$\rho_f,\rho_r>0$ and $\omega:\R^d\to\R$ is continuously
differentiable and $1$-strongly convex. We denote $\rho_F:=\rho_f+\rho_r$, then the full objective $F=f+r$ is $\rho_F\omega$-relatively weakly convex. Also assume that $F^*:=\inf_{x\in\R^d}F(x)>-\infty.$
\end{assumption}

Requiring $\omega$ to be $1$-strongly convex does not reduce relative weak convexity to ordinary weak convexity. Indeed, every $\rho$-weakly convex function is $\rho\omega$-relatively weakly convex, since $\omega-\frac12\|\cdot\|^2$ is convex, while the converse need not hold. For example, on $\R$, let $\omega(x)=\frac12x^2+x^4$ and $\varphi(x)=-x^4$. Then $\varphi+\omega=\frac12x^2$ is convex, so $\varphi$ is $\omega$-relatively weakly convex, whereas $\varphi''(x)=-12x^2$ is unbounded below, and hence $\varphi$ is not weakly convex in the Euclidean sense. This framework also includes DC functions, since if $h_1$ and $h_2$ are convex and $L\omega-h_2$ is convex, then $h_1-h_2$ is $L\omega$-relatively weakly convex. A representative example covered by \Cref{2assum:rwc} is the SCAD-regularized sparse learning model discussed in the introduction \citep{fan2001variable},
\begin{align*}
\min_{x\in\R^d}\ f(x)+\sum_{j=1}^d p_{\mathrm{SCAD}}(x_j),
\end{align*}
where $f$ is a convex, possibly nonsmooth, loss. With the Euclidean distance generating function $\omega(x)=\frac12\|x\|^2$, the function $f$ is $\rho_f\omega$-RWC for any $\rho_f>0$. Moreover, for the standard SCAD shape parameter $a>2$,
\begin{align*}
\sum_{j=1}^d p_{\mathrm{SCAD}}(x_j)+\frac{1}{2(a-1)}\|x\|^2
\end{align*}
is convex, and hence the SCAD regularizer is $\frac{1}{a-1}\omega$-RWC. Therefore, this model falls directly within the setting of \Cref{2assum:rwc}, while allowing a nonsmooth loss and a nonconvex regularizer.

The convergence analysis also requires an upper bound on the subgradient terms. We include three alternatives in the following \Cref{2assum:ub}. The first bounds the subgradients of the full objective $F$, the second allows the subgradient bound to grow with the Bregman Moreau envelope value, and the third bounds the subgradients of $f$.

\begin{assumption}\label{2assum:ub}
    Let $g_f(x) \in \partial f(x)$ and $g_F(x)\in\partial F(x)$ denote the subgradients of $f$ and $F$, respectively. We assume at least one of the following conditions holds for all $x \in \R^d$:
\begin{enumerate}[label={\textbf{A\arabic*:}}, ref={Assumption \ref{2assum:ub}.A\arabic*}]
    \item \label{2assump:comp-ubF} \textbf{Bounded subgradient of $F$:} There exists a constant $L_F>0$ such that $\|g_F(x)\|\leq L_F$ for all $g_F(x)\in \partial F(x)$. 
    \item \label{2assump:comp-ubme} \textbf{Moreau Envelope bound:} For any given $\lambda$, there exist positive scalars $\alpha_1$ and $\beta_1$ such that for all $g_F(x)\in \partial F(x)$,
	\begin{align*}
		\|g_F(x)\|\leq \sqrt{\alpha_1 [F_\lambda^\omega(x)-F^*]+\beta_1},
	\end{align*}
	where $F^*:=\inf_{x\in\R^d}F(x)$ and $F_\lambda^\omega$ is the Bregman Moreau envelope defined in \eqref{eq:def_of_me}.
    \item \label{2assump:comp-ubf} \textbf{Bounded subgradient of $f$:} There exists a constant $L_f>0$ such that $\|g_f(x)\|\leq L_f$ for all $g_f(x)\in \partial f(x)$.  
\end{enumerate}
\end{assumption}

Each alternative in \Cref{2assum:ub}, together with the additional condition stated in the main body, is sufficient for the convergence result. In the next subsection, we give the Bregman proximal subgradient method and then prove its convergence under \Cref{2assum:rwc} together with any one of the three alternatives in \Cref{2assum:ub}.

\subsection{Algorithm BPG}

We present here the deterministic Bregman proximal subgradient method. At each iteration, the method chooses a subgradient of $f$ at the current point and then solves a Bregman proximal subproblem involving the regularizer term $r$.

\begin{algorithm}[ht]
\caption{\textbf{B}regman \textbf{P}roximal sub\textbf{G}radient (BPG) method} \label{alg:comp_bpg}
{\bf Input:} Initial point $x^0\in\R^d$, step sizes $\{t_k\}_{k\geq0}$, number of iterations $K$.
\begin{algorithmic}[1]
\item {\bf for} {$k=0,1,\ldots, K-1$} {\bf do}
\item \quad Choose $g_k\in\partial f(x^k)$.
\item \quad Set
\begin{align}
\label{eq:comp-alg-update}
x^{k+1}
\in
\argmin_{x\in\R^d}
\left\{
\langle g_k,x-x^k\rangle
+r(x)
+\frac{1}{t_k}D_\omega(x,x^k)
\right\}.
\end{align}
\item {\bf end for}
\end{algorithmic}
{\bf Return} $\{x^k\}_{k=0}^{K}$.
\end{algorithm}

With \Cref{alg:comp_bpg}, we now study its convergence under \Cref{2assum:rwc} and \Cref{2assum:ub}. The convergence guarantee is formulated in terms of the Bregman stationarity measure defined in \eqref{eq:def_of_bsm}. Since \Cref{2assum:ub} contains three alternative conditions, each condition leads to a corresponding form of the convergence bound. The next subsection states these bounds together and introduces the one-step estimates used in their analysis.

\subsection{Convergence Analysis}

We now state the deterministic convergence result in the following \Cref{2thm:main_cov} for \Cref{alg:comp_bpg}. The first two cases rely on conditions that control subgradients of the full objective $F$, whereas the third case requires only bounded subgradients of $f$. This difference leads to distinct convergence bounds under the three alternatives in \Cref{2assum:ub}, but all three bounds yield an $\mathcal O(\varepsilon^{-4})$ iteration complexity for finding an $\varepsilon$-stationary point measured by \eqref{eq:def_of_bsm}.

\begin{theorem}\label{2thm:main_cov}
    Let $\{x^k\}_{k\geq0}$ be the iterates generated by \Cref{alg:comp_bpg} and let $K$ be a positive integer. Suppose that \Cref{2assum:rwc} holds. Then we have the following convergence result:
	\begin{enumerate}[label=(\alph*)]
		\item \textbf{Convergence under \ref{2assump:comp-ubF}:} For any $\lambda\in (0, \frac1{\rho_F})$ and step size sequence $\{t_k\}_{k\geq 0}\subset(0, \lambda]$, we have
		\begin{align}\label{eq:comp-estimate-ubF}
				\min_{0\le k \le K-1}\dsymomelam(\xhk, \xk) \leq \dfrac{\benvF(x^0)-F^*+\dfrac{L_F^2}{2\lambda(1-\lambda\rho_r)^2}\dsum_{k=0}^{K-1}t_k^2}{(1-\lambda\rho_F)\dsum_{k=0}^{K-1}t_k}.
		\end{align}
		In particular, if we use a constant step size $t_k= \min\left\{\dfrac{1}{2\rho_F}, \dfrac{1-\rho_r/(2\rho_F)}{L_F}\sqrt{\dfrac{\delta}{\rho_FK}}\right\}$ for some $\delta\geq F_{1/(2\rho_F)}^\omega(x^0)-F^*$, then it holds that
		\begin{align}\label{eq:comp-complexity-ubF}
				\min_{0\le k\le K-1}\dsymphic{\frac1{2\rho_F}}(\xhk, \xk) \leq 8\max\left\{\dfrac{\rho_F\delta}{K}, L_F\sqrt{\dfrac{\rho_F\delta}{K}}\right\}.
		\end{align}

        \item \textbf{Convergence under \ref{2assump:comp-ubme}:} For any $\lambda\in (0, \frac1{\rho_F})$ and $\{t_k\}_{k\geq 0}\subset(0, \lambda]$ such that $\sum_{k=0}^{K-1} t_k^2\leq C$ for some $C>0$, we have 
		\begin{align}\label{eq:comp-estimate-ubME}
				\min_{0\le k\le K-1}\dsymomelam(\xhk, \xk) \leq \dfrac{\benvF(x^0)-F^*+\dfrac{(\alpha_1C_F+\beta_1)}{2\lambda(1-\lambda\rho_r)^2}\dsum_{k=0}^{K-1}t_k^2}{(1-\lambda\rho_F)\dsum_{k=0}^{K-1}t_k},
		\end{align}
		where $C_F = \left(\benvF(x^0) - F^* + \dfrac{C\beta_1}{2(1-\lambda\rho_r)^2\lambda}\right)e^{\frac{C\alpha_1}{2(1-\lambda\rho_r)^2\lambda}}.$
		In particular, for a fixed $\Cbar>0$, using a constant step size $t_k= \min\left\{\dfrac{1}{2\rho_F}, \sqrt{\frac{\left(1-\frac{\rho_r}{2\rho_F}\right)^2\delta}{\rho_F(\alpha_1\Cbar_F+\beta_1)K}}, \sqrt{\frac{\Cbar}{K}}\right\}$ for some $\delta\geq F_{1/(2\rho_F)}^\omega(x^0)-F^*$, it holds that
		\begin{align}
				\min_{0\le k\le K-1}\dsymphic{1/(2\rho_F)}(\xhk, \xk) \leq 8\max\left\{\dfrac{\rho_F\delta}{K}, \sqrt{\dfrac{(\alpha_1\Cbar_F+\beta_1)\rho_F\delta}{K}}, \frac{\delta}{2\sqrt{\Cbar K}}\right\},\label{eq:comp-complexity-ubME}
		\end{align}
		where $\Cbar_F = (\delta + 4\Cbar\beta_1\rho_F)e^{4\Cbar\alpha_1\rho_F}$.

		\item \textbf{Convergence under \ref{2assump:comp-ubf}:} Assume $x^0\in\dom r$. For any $\lambda\in(0,\frac1{\rho_F})$ and any nonincreasing step size sequence $\{t_k\}_{k\geq0}$ satisfying $1-t_k\rho_r>0$, we have
		\begin{align}\label{eq:comp-estimate-ubf}
		\min_{0\le k\le K-1}\dsymomelam(\xhk,\xk)\leq\dfrac{\benvF(x^0)-F^*+\dfrac{t_0}{\lambda(1-t_0\rho_r)}\left(F(x^0)-F^*\right)+\dfrac{2L_f^2}{\lambda}\sum_{k=0}^{K-1}\dfrac{t_k^2}{1-t_k\rho_r}}
        {(1-\lambda\rho_F)\sum_{k=0}^{K-1}\dfrac{t_k}{1-t_k\rho_r}}.
		\end{align}
		In particular, let $\lambda=\dfrac{1}{2\rho_F}$ and use a constant step size $t_k= \min\left\{\dfrac{1}{2\rho_F}, \dfrac1{2L_f}\sqrt{\dfrac{\delta}{\rho_FK}}\right\}$ for some $\delta$ satisfying
		\begin{align*}
		\delta\geq\max\left\{F_{1/(2\rho_F)}^\omega(x^0)-F^*,F(x^0)-F^*\right\}.
		\end{align*}
		Then it holds that
		\begin{align}\label{eq:comp-complexity-ubf}
				\min_{0\le k\le K-1}\dsymphic{1/(2\rho_F)}(\xhk,\xk) \leq 8\max\left\{\dfrac{3\rho_F\delta}{K}, 4L_f\sqrt{\dfrac{\rho_F\delta}{K}}\right\}.
		\end{align}
	\end{enumerate}
\end{theorem}

The full proof of \Cref{2thm:main_cov} is deferred to \Cref{sec:pf_of_2thm_main}, and the supporting technical lemmas are provided in \Cref{sec:pf_2}. 
As the foundation of this theorem, we first need to establish the descent lemmas of the algorithm. The following \Cref{lem:comp-lem-rec-gF} controls the Bregman Moreau envelope using a subgradient of the full objective $F$.

\begin{lemma}\label{lem:comp-lem-rec-gF}
	Under \Cref{2assum:rwc}, for any $\lambda\in (0, \frac1{\rho_F})$ and step size sequence $\{t_k\}_{k\geq 0}\subset(0, \lambda]$, the iterates $\{x^k\}_{k\geq0}$ generated by Algorithm \ref{alg:comp_bpg} satisfy
	\begin{align}\label{eq:comp-lem-gF}
			\benvF(\xkp) -  \benvF(\xk)
			\leq&\ - \dfrac{t_k(1-\lambda \rho_F)}{(1-t_k\rho_r)\lambda^2}\(D_\omega(\xhk,\xk) + D_\omega(\xk, \xhk)\) + \dfrac{t_k^2}{2(1-t_k\rho_r)^2\lambda} \|g_F(\xk)\|^2,
	\end{align}
    where $\benvF$ is the Moreau envelope defined in \eqref{eq:def_of_me}. 
\end{lemma}
The proof of \Cref{lem:comp-lem-rec-gF} is provided in \Cref{sec:pf_of_2lem_2}.
While the descent lemma guaranteed by \Cref{lem:comp-lem-rec-gF} requires a bound on the subgradient of the entire objective $F$, \Cref{lem:comp-lem-rec-gf} applies when only the subgradients of $f$ are bounded. Its proof is provided in \Cref{sec:pf_of_2lem_1}.

\begin{lemma}
\label{lem:comp-lem-rec-gf}
Suppose that \Cref{2assum:rwc} and \ref{2assump:comp-ubf} hold. Fix $\lambda\in(0,\frac{1}{\rho_F})$ and suppose $1-t_k\rho_r>0$. Then the iterates generated by \Cref{alg:comp_bpg} satisfy
\begin{align}
\label{eq:comp-lem-gf}
\frac{t_k(1-\lambda\rho_F)}{1-t_k\rho_r}\dsymomelam(\hat x^k,x^k)
\leq F_\lambda^\omega(x^k)-F_\lambda^\omega(x^{k+1})+\frac{t_k}{\lambda(1-t_k\rho_r)}\left(F(x^k)-F(x^{k+1})\right)+\frac{2t_k^2L_f^2}{\lambda(1-t_k\rho_r)}.
\end{align}
\end{lemma}

The two lemmas reflect the two ways the deterministic analysis can be closed. When a bound on $\partial F$ is available, \Cref{lem:comp-lem-rec-gF} applies directly. When only the subgradients of $f$ are controlled, \Cref{lem:comp-lem-rec-gf} avoids requiring a bound on $\partial F$. This completes the deterministic analysis of the BPG update, and in the next section we extend the same Bregman framework to stochastic model-based updates.


\section{Bregman Stochastic Model-Based Minimization}\label{sec:BSMM}

We now consider the stochastic composite problem
\begin{align}
\label{eq:sto-comp-opt}
\min_{x\in\R^d}\ F(x):=f(x)+r(x),
\qquad
f(x):=\E_{\xi\sim P}[f(x;\xi)] .
\end{align}
Here $\xi\sim P$ is sampled from a fixed distribution, $f(\cdot;\xi):\R^d\to\R$ is proper and lower semicontinuous for almost every $\xi$, and $r:\R^d\to\R$ is a deterministic regularizer.

The deterministic BPG method in the previous section uses a subgradient of $f$ at each iteration. Its most direct stochastic extension is therefore to replace this deterministic subgradient with a stochastic subgradient oracle. Such an extension leads to the stochastic proximal subgradient method, but it does not cover the full sampled function used in stochastic proximal point methods or structured local approximations used in stochastic prox-linear methods. To accommodate these methods within the same framework, we follow the stochastic model-based formulation of \citet{davis2019stochastic}. Given a base point $x$, let $f_x(\cdot;\xi)$ denote a stochastic model of $f$. At iteration $k$, after sampling $\xi_k$, we compute
\begin{align}
\label{eq:sto-intro-bsmm-update}
x^{k+1}\in\argmin_{y\in\R^d}\left\{f_{x^k}(y;\xi_k)+r(y)+\frac{1}{t_k}D_\omega(y,x^k)
\right\}.
\end{align}
This formulation also covers stochastic proximal point, stochastic proximal subgradient, and stochastic prox-linear methods in the same notation \citep{davis2019stochastic,asi2019stochastic,davis2019proximally,drusvyatskiy2018error,zhang2022stochastic}. 

In this section we analyze the stochastic model-based minimization \eqref{eq:sto-intro-bsmm-update} in the previous Bregman setting where $f(\cdot)$ and $r(\cdot)$ are only relatively weakly convex. The rest of this section is organized as follows. In Section~\ref{sec:3.1}, we state the assumptions on the stochastic models and show that they imply relative weak convexity of the expected objective $F$. In Section~\ref{sec:3.2}, we present the Bregman stochastic model-based algorithm. In Section~\ref{sec:3.3}, we prove the convergence guarantee and derive an $\mathcal O(\varepsilon^{-4})$ iteration complexity for finding an expected $\varepsilon$-stationary point.

\subsection{Preliminaries and Assumptions}
\label{sec:3.1}

We next state the assumptions on the stochastic models $f_x(\cdot;\xi)$ in \eqref{eq:sto-intro-bsmm-update}. These conditions are stated directly in terms of the model functions rather than stochastic subgradients, and they extend the relative weak convexity and subgradient-control conditions from the deterministic setting to stochastic model functions, while adding a one-sided accuracy condition that relates the models to the expected loss $f$.

\begin{assumption}\label{asp:model-basic}\leavevmode
Fix a probability space $(\Omega,\mathcal F,P)$ and equip $\R^d$ with the Borel $\sigma$-algebra. Assume that $r$ is proper and lower semicontinuous, that $(x,y,\xi)\mapsto f_x(y;\xi)$ is jointly measurable, and that $f_x(\cdot;\xi)+r(\cdot)$ is proper and lower semicontinuous for every $x\in\R^d$ and almost every $\xi\in\Omega$. Also assume there exist constants $\tau,\rho>0$ and a continuously differentiable, $1$-strongly convex function $\omega:\R^d\to\R$ such that $F^*:=\inf_{x\in\R^d}F(x)>-\infty$ and, for every $x \in \R^d$:
\begin{enumerate}[label={\textbf{A(\arabic*):}}, ref={Assumption \ref{asp:model-basic}.A(\arabic*)}]
\item \label{asp:model-unbiased} \textbf{One-sided accuracy.}
For all $x, y\in\R^d$,
\begin{align}
\E_\xi[\sfx(x;\xi)] &= f(x), \qquad
\E_\xi[\sfx(y;\xi)-f(y)] \leq
\tau\min\{D_\omega(x,y),D_\omega(y,x)\}. \label{eq:model-asp-unbiased}
\end{align}
\item \label{asp:model-rwc} \textbf{Relative weak convexity.}
For every $x\in \R^d$ and  $\xi\in\Omega$, the stochastic composite surrogate $\sfx(\cdot;\xi)+r(\cdot)$
is $\rho\omega$-relatively weakly convex.
\item \label{asp:model-lip} \textbf{Relative Lipschitz continuity.}
There exist a measurable function $\sml:\R^d\times\Omega\to\R_+$ and a deterministic function $L:\R^d\to\R_+$ such that
\begin{align}
\sqrt{\E_\xi[\sml^2(x;\xi)]}\leq L(x)
\end{align}
and, for all $y\in\R^d$ and almost every $\xi\in\Omega$,
\begin{align}
\sfx(x;\xi)-\sfx(y;\xi)
\leq
\sml(x;\xi)
\sqrt{\min\{D_\omega(x,y),D_\omega(y,x)\}} .
\end{align}
\end{enumerate}
\end{assumption}

\ref{asp:model-unbiased} requires the stochastic surrogate to be unbiased at the center point $x$ and models the one-sided approximation error via the Bregman distance. \ref{asp:model-rwc} ensures that the stochastic subproblem objective inherits weak convexity relative to $\omega$. \ref{asp:model-lip} controls the continuity of the surrogate measured by the Bregman distance.
The corresponding growth parameter $L(x)$ is assumed to satisfy one of the following two conditions.

\begin{assumption}\label{asp:ub}
The growth parameter $L(x)$ satisfies at least one of the following:
\begin{enumerate}[label={\textbf{A(\arabic*):}}, ref={Assumption \ref{asp:ub}.A(\arabic*)}]
\item \label{asp:model-lip-f} \textbf{Global uniform bound.}
There exists $L_f>0$ such that $L(x)\leq L_f$
for all $x\in\R^d$.
\item \label{asp:model-lip-ME} \textbf{Moreau envelope-dependent growth.}
There exist $\alpha_1,\beta_1>0$ such that
\begin{align}
L(x)\leq\sqrt{\alpha_1\bigl(F_\lambda^\omega(x)-F^*\bigr)+\beta_1}
\end{align}
for all $x\in\R^d$.
\end{enumerate}
\end{assumption}

Since the stationarity measure \eqref{eq:def_of_bsm} is defined with respect to the true objective $F=f+r$ rather than individual local surrogates $f_x(\cdot; \xi)$, it is also necessary to obtain the relative weak convexity of $F$ inherited from its components to guarantee that \eqref{eq:def_of_bsm} is well-defined. The following proposition formalizes this property.

\begin{proposition}\label{lem:model-exp-rwc}
If \Cref{asp:model-basic} holds, then $F=f+r$ is $(\tau+\rho)\omega$-relatively weakly convex.
\end{proposition}

With the relative weak convexity of $F$ established, the task of finding an approximate stationary point for the composite problem \eqref{eq:sto-comp-opt} is well-posed. We now proceed to introduce the algorithmic framework, which iteratively minimizes these stochastic surrogates to achieve this goal.

\subsection{Algorithm BSMM}
\label{sec:3.2}
We state the \textbf{B}regman \textbf{S}tochastic \textbf{M}odel-Based \textbf{M}inimization (BSMM) framework in \Cref{alg:bsmm}. 
At each iteration, the update step \eqref{eq:3subpro} minimizes the stochastic model surrogate $f_{x^k}(\cdot; \xi_k)$ alongside the deterministic regularizer $r(\cdot)$ and the Bregman distance $D_\omega(\cdot,x^k)$. Following standard conventions in stochastic nonconvex optimization, the final output $x^{\bar{k}}$ is drawn randomly from the generated trajectory according to a discrete distribution proportional to the step sizes. 

\begin{algorithm}[ht]
	\caption{Bregman Stochastic Model-based Minimization (BSMM)} \label{alg:bsmm}
	{\bf Input:}  Initial point $x^0\in\dom r$, step sizes $\{t_k\}_{k\geq0}>0$, and maximum number of iterations $K$.
	\begin{algorithmic}[1]
		\item[1:] {\bf for} {$k=0,1,\ldots, K-1$} {\bf do}
		\item[2:] $\quad$Sample $\xi_k \sim P$;
		\item[3:] $\quad$Update  
        \begin{align}\label{eq:3subpro}
            x^{k+1} = \argmin_{y\in\R^d} \left\{f_{x^k} (y;\xi_k) + r(y) + \frac{1}{t_k} D_\omega (y,x^k)\right\};
        \end{align}
		\item[4:] {\bf end for}
		\item[5:] Sample $\bar{k}$ from $\{0,1,\ldots, K-1\}$ according to the discrete probability distribution
        \begin{align*}
            \mathbb P(\bar k=k)=\frac{t_k}{\sum_{i=0}^{K-1}t_i}.
        \end{align*}
	\end{algorithmic}
    {\bf Return} $x^{\bar{k}}$.
\end{algorithm}

With the algorithmic framework established, the remaining question is whether BSMM can find an approximate stationary point of the composite problem and how many iterations are required. \Cref{3thm:main} answers this question by establishing convergence bounds for the expected Bregman stationarity measure and the corresponding iteration complexity.
\subsection{Convergence Analysis}
\label{sec:3.3}

We now present the convergence guarantees for \Cref{alg:bsmm}. In \Cref{3thm:main}, we demonstrate that the iterate sequence of \Cref{alg:bsmm} converges at a rate of $\mathcal{O}(\varepsilon^{-4})$. The complete proof is provided in \Cref{sec:pf_of_3thm_main}.

\begin{theorem}\label{3thm:main}
    Suppose that \Cref{asp:model-basic} holds. Let $\lambda\in(0,\frac1{\tau+\rho})$, and let $\{t_k\}_{k\geq0}$ be a nonincreasing step size sequence satisfying $0<t_k<\min\{\frac1{\rho},\frac1{\tau}\}$.
    \begin{enumerate}[label=(\alph*)]
        \item \textbf{Convergence under \ref{asp:model-lip-f}:} If \ref{asp:model-lip-f} holds, then
        \begin{align}\label{eq:model-thm-estimate}
        \E[\dsymomelam(\xhat^{\kbar},\xkbar)]\leq\frac{\benvF(x^0)-F^*+\dfrac{t_0}{\lambda(1-t_0\rho)}\left(F(x^0)-F^*\right)+\dfrac{L_f^2}{\lambda}\sum_{k=0}^{K-1}\dfrac{t_k^2}{(1-t_k\rho)(1-t_k\tau)}}{(1-\lambda(\tau+\rho))\sum_{k=0}^{K-1}t_k}.
        \end{align}
        In particular, let $\lambda=\frac1{2(\tau+\rho)}$ and $t_k=\frac{c}{\sqrt K}$, where $0<c\leq\min\{\frac1{2\rho},\frac1{2\tau}\}$. For any $\delta\geq\max\{\benvF(x^0)-F^*,F(x^0)-F^*\}$, we have
        \begin{align}\label{eq:model-thm-complexity-ubf}
        \E[\dsymphic{1/(2(\tau+\rho))}(\xhat^{\kbar},\xkbar)]\leq\frac{2}{c\sqrt K}\left(\delta+\frac{4(\tau+\rho)c\delta}{\sqrt K}+8(\tau+\rho)L_f^2c^2\right).
        \end{align}

        \item \textbf{Convergence under \ref{asp:model-lip-ME}:} Suppose that \ref{asp:model-lip-ME} holds and $\frac{\alpha_1}{\lambda}\sum_{k=0}^{K-1}\frac{t_k^2}{(1-t_k\rho)(1-t_k\tau)}<1.$
        Then
        \begin{align}\label{eq:model-thm-estimate-ME}
        \E[\dsymomelam(\xhat^{\kbar},\xkbar)]\leq&\ \frac{1}{(1-\lambda(\tau+\rho))\left(1-\dfrac{\alpha_1}{\lambda}\sum_{k=0}^{K-1}\dfrac{t_k^2}{(1-t_k\rho)(1-t_k\tau)}\right)\sum_{k=0}^{K-1}t_k}\Bigg\{\benvF(x^0)-F^*\nonumber\\
        &+\frac{t_0}{\lambda(1-t_0\rho)}\left(F(x^0)-F^*\right)+\frac{\beta_1}{\lambda}\sum_{k=0}^{K-1}\frac{t_k^2}{(1-t_k\rho)(1-t_k\tau)}\Bigg\}.
        \end{align}
        In particular, let $\lambda=\frac1{2(\tau+\rho)}$ and $t_k=\frac{c}{\sqrt K}$, where $0<c\leq\min\left\{\frac1{2\rho},\frac1{2\tau}\right\},\  8\alpha_1(\tau+\rho)c^2<1.$
        For any $\delta\geq\max\{\benvF(x^0)-F^*,F(x^0)-F^*\}$, we have
        \begin{align}\label{eq:model-thm-complexity-ME}
        \E[\dsymphic{1/(2(\tau+\rho))}(\xhat^{\kbar},\xkbar)]\leq\frac{2}{c\sqrt K}\frac{\delta+\dfrac{4(\tau+\rho)c\delta}{\sqrt K}+8(\tau+\rho)\beta_1c^2}{1-8\alpha_1(\tau+\rho)c^2}.
        \end{align}
    \end{enumerate}
\end{theorem}
The proof of \Cref{3thm:main} relies on the following one-step expected descent inequality in \Cref{lem:model-recursion}, which extends the deterministic analysis from Section~\ref{sec:BPG} to the stochastic surrogate. The proof is in \Cref{sec:pf_of_3lem_2}.

\begin{lemma}\label{lem:model-recursion}
    Suppose that \Cref{asp:model-basic} holds. Let $\lambda\in(0,\frac1{\tau+\rho})$ and $0<t_k<\min\{\frac1{\rho},\frac1{\tau}\}$. Then the iterates generated by \Cref{alg:bsmm} satisfy
    \begin{align}
    \frac{t_k(1-\lambda(\tau+\rho))}{1-t_k\rho}\dsymomelam(\xhk,\xk)\leq&\ \benvF(\xk)-\E_k[\benvF(\xkp)]+\frac{t_k}{\lambda(1-t_k\rho)}\left(F(\xk)-\E_k[F(\xkp)]\right)\nonumber\\
    &+\frac{t_k^2}{2\lambda(1-t_k\rho)(1-t_k\tau)}\left(\E_k[L^2(\xkp)]+L^2(\xk)\right).\label{eq:model-recursion}
    \end{align}
\end{lemma}

Summing \eqref{eq:model-recursion} over $k$ and substituting the prescribed parameter choices yields \Cref{3thm:main}. This completes the analysis of BSMM with sampled objective models and an exactly evaluated Bregman distance.

\section{Model-Based Minimization with Stochastic Bregman Distances} \label{sec:SBMM}

The BSMM framework in Section~\ref{sec:BSMM} allows the objective model to be stochastic but assumes that the Bregman distance $D_\omega$ is evaluated exactly. This assumption is reasonable when the distance generating function $\omega$ has a simple closed form. However, when $\omega$ is itself represented as an expectation or a large finite sum, evaluating $\omega$ and $\nabla\omega$ exactly at every iteration can remove much of the computational advantage gained from sampling the objective model. We therefore extend the previous framework by allowing the reference geometry to be sampled as well.

Suppose that we have access to stochastic realizations of the distance generating function and its gradient, denoted by $\some(\cdot;\xi)$ and $\snome(\cdot;\xi)$, respectively. For a sample $\xi$, define the corresponding stochastic Bregman distance by
\begin{align*}
    \sdome(x,y;\xi) \triangleq \some(x;\xi) - \some(y;\xi) - \la \snome(y;\xi), x-y\ra.
\end{align*}

At each iteration $k$, we draw a minibatch of i.i.d. samples $\Bcal_k \triangleq\{\xi_{k,1}, \ldots, \xi_{k,m_k}\}$ of size $m_k$. We define the stochastic approximations of the objective model surrogate and the Bregman distance as the average
\begin{align*}
    \sfxk(x;\Bcal_k) \triangleq \frac{1}{m_k}\sum_{i=1}^{m_k}\sfxk(x;\xi_{k,i}), \qquad \sdome (x, y;\Bcal_k) \triangleq \frac{1}{m_k}\sum_{i=1}^{m_k} \sdome(x,y;\xi_{k,i}).
\end{align*}
We also define \begin{align*}
\some(x;\Bcal_k):=\frac1{m_k}\sum_{i=1}^{m_k}\some(x;\xi_{k,i}),\qquad \snome(x;\Bcal_k):=\frac1{m_k}\sum_{i=1}^{m_k}\snome(x;\xi_{k,i}).
\end{align*} Using these minibatch estimators, we perform the following update
\begin{align}\label{eq:sbmm_update}
    x^{k+1} = \argmin_{y\in\R^d} \left\{ \sfxk(y;\Bcal_k) + r(y) + \frac{1}{t_k} \sdome(y, x^k; \Bcal_k) \right\}.
\end{align}
We refer to this procedure as the \textbf{S}tochastic \textbf{B}regman \textbf{M}odel-based \textbf{M}inimization (SBMM) method.
In addition, we define the full stochastic composite model as $\sFxk(x;\Bcal_k) \triangleq \sfxk(x;\Bcal_k) + r(x)$, and denote $\E_k[\cdot] \triangleq \E[\cdot \mid \sigma_k]$, where $\sigma_k$ is the $\sigma$-algebra generated by the history of minibatches $\{\Bcal_i\}_{i=0}^{k-1}$.

\subsection{Preliminaries and Assumptions}
Here we state the basic conditions of the stochastic objective model and the sampled distance generating function. These conditions extend the model assumptions of Section~\ref{sec:BSMM} by requiring the Bregman geometry and the relative weak convexity property to hold for each stochastic sample, while preserving the corresponding deterministic quantities in expectation. The additional conditions needed to control the composite
model and the sampled Bregman distance along the iterates are also introduced separately below.

\begin{assumption}\label{asp:model-sb-basic}
    Fix a probability space $(\Omega,\mathcal F,P)$ and equip $\R^d$ with the Borel $\sigma$-algebra. Assume that $r$ is proper and lower semicontinuous, that $(x,y,\xi)\mapsto f_x(y;\xi)$ is jointly measurable, and that $f_x(\cdot;\xi)+r(\cdot)$ is proper and lower semicontinuous for every $x\in\R^d$ and almost every $\xi\in\Omega$. Suppose there exist constants $\rho,\tau>0$ and a differentiable function $\omega:\R^d\to\R$. We further assume that $F^*:=\inf_{x\in\R^d}F(x)>-\infty$.
    \begin{enumerate}[label={\textbf{A(\arabic*):}}, ref={Assumption \ref{asp:model-sb-basic}.A(\arabic*)}]
        \item \label{asp:model-sb-dgf-unbiased} \textbf{Unbiased DGF:}
The maps $(x,\xi)\mapsto\some(x;\xi)$ and $(x,\xi)\mapsto\snome(x;\xi)$ are jointly measurable. For almost every $\xi\in\Omega$, the function $\some(\cdot;\xi)$ is differentiable and $1$-strongly convex, with $\snome(x;\xi)=\nabla\some(x;\xi)$. Moreover,
\begin{align*}
\E_\xi[|\some(x;\xi)|]+\E_\xi[\|\snome(x;\xi)\|]<\infty,\qquad
\E_\xi[\some(x;\xi)]=\omega(x),\qquad
\E_\xi[\snome(x;\xi)]=\nabla\omega(x).
\end{align*}
Consequently, $\E_\xi[\sdome(x,y;\xi)]=D_\omega(x,y)$.
        
        \item \label{asp:model-sb-f} \textbf{One-sided accuracy:} For any $x,y\in\R^d$,
        \begin{align*}
            &\E_\xi[\sfx(x;\xi)] =  f(x), \qquad
            &\E_\xi[\sfx(y;\xi)-f(y)] \leq \tau \min\{D_\omega(x,y), D_\omega(y, x)\}. 
        \end{align*}
        
        \item \label{asp:model-sb-rwc} \textbf{Relative Weak Convexity:} For any $x\in\rr^d$ and almost every $\xi\in\Omega$, the function $\sfx(\cdot; \xi) + r(\cdot)$ is $\rho\some(\cdot;\xi)$-relatively weakly convex.
    \end{enumerate}
\end{assumption}

The first part of \Cref{asp:model-sb-basic} makes $\sdome(\cdot,\cdot;\xi)$ an unbiased stochastic counterpart of $D_\omega$. The second part is the same one-sided model accuracy condition used in Section~\ref{sec:BSMM}. The third part changes the weak convexity requirement from the deterministic geometry $\omega$ to the sampled geometry $\some(\cdot;\xi)$, which matches the Bregman term used in \eqref{eq:sbmm_update}.

The next assumption is the Lipschitz-type condition on the composite model. Since the subproblem contains both $f_{x^k}(\cdot;\xi)$ and $r$, the condition is imposed on the full composite model rather than on the objective model alone.

\begin{assumption}\label{asp:model-sb-lip}
There exist a measurable function $\sml:\R^d\times\Omega\to\R_+$ and a constant $L>0$ such that
\begin{align*}
\sqrt{\E_\xi[\sml^2(x;\xi)]}\leq L
\end{align*}
and, for all $x,y\in\R^d$ and almost every $\xi\in\Omega$,
\begin{align}
\sFx(x;\xi)-\sFx(y;\xi)
\leq
\sml(x;\xi)
\sqrt{\min\{\sdome(x,y;\xi),\sdome(y,x;\xi)\}}.
\label{eq:model-sb-asp-lip-left}
\end{align}
\end{assumption}

\begin{remark}
\Cref{asp:model-sb-basic} and \Cref{asp:model-sb-lip} introduce the differences compared to our stochastic model surrogate setting in \Cref{asp:model-basic} and \Cref{asp:ub} in the following two ways:
\begin{itemize}
    \item The relatively weak convexity property must now hold with respect to the \textit{stochastic} DGF $\some(\cdot;\xi)$, rather than the deterministic $\omega(\cdot)$. This ensures that the randomized subproblems in \eqref{eq:sbmm_update} remain well-defined at every iteration.
    \item Due to coupled noise in both the geometry and the objective, we require stronger relative Lipschitz conditions on the \textit{entire} composite model $\sfx+r$, rather than on $\sfx$ alone.
\end{itemize} 
\end{remark}

The assumptions above define a valid subproblem \eqref{eq:sbmm_update}, but they do not by themselves compare the sampled distance with the exact Bregman distance, which is needed since the stationarity measure is still defined using $D_\omega$. We therefore impose the following lower bound condition along the iterates.

\begin{assumption}[Lower Bound on Stochastic Bregman Distance] \label{asp:model-sb-sbg}
    For each iteration $k$ with minibatch size $m_k$, let $\Bcal_k$ be a minibatch and let $x^{k+1}$ be the resulting SBMM update. There exists a function $c:\mathbb N\to(0,1]$ satisfying $\lim_{m\to\infty}c(m)=1$ such that
\begin{align}\label{eq:sbg_conditional}
    \E_k\left[\sdome(\hat x^k,x^{k+1};\Bcal_k)\right]\geq c(m_k)\E_k\left[D_\omega(\hat x^k,x^{k+1})\right].
\end{align}
\end{assumption}
The role of \Cref{asp:model-sb-sbg} is to compare the sampled Bregman distance with the exact Bregman distance along the iterates. In practice, a sufficiently large batch size $m_k$ allows $c(m_k)$ to approach $1$, so that the sampled Bregman distance approaches the exact Bregman distance in the sense of \eqref{eq:sbg_conditional}. Since verifying this assumption directly can be non-trivial due to the dependence between the iterates and the minibatches, we note that it is naturally implied if we assume uniform expectation lower bounds for fixed minibatches, or high-probability geometric concentration bounds across the domain space.

\begin{remark}[Sufficient conditions for \Cref{asp:model-sb-sbg}]
A stronger condition that implies \Cref{asp:model-sb-sbg} is a uniform lower bound on a set containing the iterates. Suppose there exists a set $\mathcal X\subseteq\R^d$ such that, for each minibatch size $m$ and every minibatch $\Bcal$ of size $m$,
\begin{align*}
\sdome(u,v;\Bcal)\geq c(m)D_\omega(u,v)\qquad\text{for all }u,v\in\mathcal X.
\end{align*}
Then \Cref{asp:model-sb-sbg} follows by applying this inequality with $u=\hat x^k$ and $v=x^{k+1}$ and taking conditional expectation.

As a second sufficient condition, \Cref{asp:model-sb-sbg} can be verified through a high-probability lower bound. Suppose that for each minibatch size $m$, there exist $\epsilon_m,\theta_m\in[0,1)$ and an event $\mathcal G_m$ such that
\begin{align*}
\sdome(u,v;\Bcal)\geq(1-\epsilon_m)D_\omega(u,v)
\qquad\text{for all }u,v\in\mathcal X \quad\text{on }\mathcal G_m,
\end{align*}
and
\begin{align*}
\E_k\left[D_\omega(\hat x^k,x^{k+1})\mathbf 1_{\mathcal G_m^c}\right]
\leq \theta_m \E_k\left[ D_\omega(\hat x^k,x^{k+1})\right].
\end{align*}
Therefore, if $(1-\epsilon_m)(1-\theta_m)\to 1$, then \Cref{asp:model-sb-sbg} holds with $c(m)=(1-\epsilon_m)(1-\theta_m).$
\end{remark}

These assumptions ensure that the randomized subproblem in \eqref{eq:sbmm_update} is well-defined and that the error introduced by sampling the Bregman distance remains bounded. With these conditions in place, we now state the formal algorithm.

\subsection{Algorithm SBMM}
We now present the formal minibatch version of the update in \eqref{eq:sbmm_update} in \Cref{alg:sbmm}. At each iteration, the same minibatch is used to form both the stochastic model and the sampled Bregman distance.

\begin{algorithm}[h]
    \caption{Stochastic Bregman Model-based Minimization (SBMM)}
    \label{alg:sbmm}
    {\bf Input:}  Initial point $x^0\in\dom r$, step sizes $\{t_k\}_{k\geq0}$, batch sizes $\{m_k\}_{k\geq 0}$, and maximum iterations $K$.
	\begin{algorithmic}[1]
		\item[1:] {\bf for} {$k=0,1,\ldots, K-1$} {\bf do}
		\item[2:] $\quad$Sample a minibatch $\Bcal_k =\{\xi_{k,1}, \xi_{k,2}, \ldots, \xi_{k,m_k}\}$;
		\item[3:] $\quad$Update  
		\begin{align*}
            \xkp = \argmin\limits_y \frac1{m_k}\sum_{i=1}^{m_k}\left(\sfxk (y;\xi_{k,i}) + r(y) + \frac{1}{t_k} \sdome (y,\xk;\xi_{k,i})\right);
        \end{align*}
		\item[4:] {\bf end for}
		\item[5:] Sample $\kbar$ from $\{0,1,\ldots, K-1\}$ according to the discrete distribution:
        \begin{align*}
            \mathbb P(\kbar=k)=\frac{t_k}{\sum_{i=0}^{K-1}t_i}.
        \end{align*}
	\end{algorithmic}
    {\bf Return} $x^{\bar{k}}$.
\end{algorithm}

At each step, the algorithm constructs the objective surrogate and the reference geometry using the average over the current minibatch $\Bcal_k$, then evaluates the update by minimizing this stochastic composite function. The final output is randomly selected in iterates $x^{\bar{k}}$, where the selection probability is proportional to the step size $t_k$. 

With the SBMM algorithm and the structural assumptions on the stochastic oracles, we now proceed to establish its theoretical properties. The following section analyzes the convergence of the sequence generated by \Cref{alg:sbmm} and bounds the expected stationarity measure.

\subsection{Convergence Analysis}

We now establish the convergence properties of \Cref{alg:sbmm}. \Cref{thm:model-sbmm} demonstrates that under an appropriate step size and a sufficiently large minibatch, \Cref{alg:sbmm} obtains the standard $\mathcal{O}(\varepsilon^{-4})$ convergence rate, accommodating the variance introduced by the sampled Bregman distance.

\begin{theorem}\label{thm:model-sbmm}
    Suppose that \Cref{asp:model-sb-basic}, \Cref{asp:model-sb-sbg} and \Cref{asp:model-sb-lip} hold. Fix $\lambda\in\left(0,\frac{1}{\rho+\tau}\right)$. Assume that $m_k\equiv m$ and $t_k\equiv t$, where $\frac{\lambda(1-c(m))}{1-\lambda\tau-\lambda\rho c(m)}<t<\frac1\rho.$
    Then the point $x^{\bar k}$ returned by \Cref{alg:sbmm}, where $\mathbb P(\bar k=k)=1/K$, satisfies
    \begin{align}\label{eq:sbmm_final_bound_explicit}
    \E\left[\dsymomelam(\hat x^{\bar k},x^{\bar k})\right]\leq\frac{\benvF(x^0)-F^*}{\Gamma K}+\frac{t^2L^2}{4\lambda c(m)(1-t\rho)\Gamma},
    \end{align}
    where $\Gamma:=\dfrac{1}{c(m)}\left(\dfrac{t(1-\lambda(\rho+\tau))}{1-t\rho}-\lambda(1-c(m))\right).$
    In particular, let $K>1$, set $t=\frac{1}{\rho\sqrt K}$, and suppose that the minibatch size $m=m(K)$ satisfies $1-c(m)\leq\frac{1-\lambda(\rho+\tau)}{2\lambda\rho\sqrt K}$, then
    \begin{align}\label{eq:sbmm_rate_O1sqrtK_explicit}
    \E\left[\dsymomelam(\hat x^{\bar k},x^{\bar k})\right]\leq\frac{1}{\sqrt K}\left(\frac{2\rho\left(\benvF(x^0)-F^*\right)}{1-\lambda(\rho+\tau)}+\frac{L^2}{2\lambda\rho\left(1-\lambda(\rho+\tau)\right)}\right).
    \end{align}
\end{theorem}

The main convergence result relies on a one-step descent lemma. The following \Cref{lem:model-sb-1} bounds the expected decrease of the Moreau envelope along the algorithm's trajectory, serving as the foundational step for the proof of \Cref{thm:model-sbmm}.

\begin{lemma}\label{lem:model-sb-1}
    Suppose that \Cref{asp:model-sb-basic}, \Cref{asp:model-sb-sbg} and \Cref{asp:model-sb-lip} hold. Let $\lambda\in\left(0,\frac{1}{\rho+\tau}\right)$ and $0<t_k<1/\rho$. Then
    \begin{align}\label{eq:lem_model_sb_1}
    &\quad\ \E_k\left[\benvF(\xkp)\right]\nonumber\\
    \leq&\ \benvF(\xk)-\frac{1}{c(m_k)}\left(\frac{t_k(1-\lambda(\rho+\tau))}{1-t_k\rho}-\lambda(1-c(m_k))\right)\dsymomelam(\xhk,\xk)+\frac{t_k^2L^2}{4\lambda c(m_k)(1-t_k\rho)}.
    \end{align}
\end{lemma}

Lemma~\ref{lem:model-sb-1} provides the recursive relationship for the conditional expectation of the Moreau envelope across successive iterations. Summing this inequality over the optimization horizon balances the accumulated objective descent against the cumulative variance from both the objective surrogate and the stochastic reference geometry. This direct summation yields the explicit bounds stated in Theorem~\ref{thm:model-sbmm}, establishing that the stochastic Bregman geometry can be tightly integrated into model-based minimization under the minibatch lower-bound condition.

\section{Numerical Experiments}
In this section, we present numerical experiments to validate our theoretical findings and evaluate the empirical performance of the proposed methods. Specifically, we apply BPG to the $\ell_1-\ell_p$ regularized compressed sensing problem and SBMM to the robust $L_p$ PCA problem. We compare the proposed methods with standard problem-specific baselines. The experiments are conducted  on a MacBook Pro with an Apple M3 Pro chip. The compressed sensing experiments are implemented in MATLAB
R2024b, while the PCA experiments are implemented in Julia. 

\subsection{Compressed Sensing}
We first apply the Bregman proximal subgradient method to the compressed sensing problem in \cite{wang2010sparse}. 
Consider the following $\ell_1-\ell_p$ regularized recovery problem:
\begin{align*}
    \min_{x\in\rr^n} \frac{1}{2}\|Ax-b\|_2^2 + \lambda_1\|x\|_1 - \lambda_2 \|x\|_p.
\end{align*}

\paragraph{Algorithm.}
We apply \Cref{alg:comp_bpg} by setting the objective surrogate as $f(x) = \frac{1}{2}\|Ax-b\|_2^2$ and the regularizer as $r(x) = \lambda_1\|x\|_1 - \lambda_2 \|x\|_p$. The step sizes are chosen as $t_k \equiv t = \lambda/\sqrt{\nu}$, and the distance generating function is defined as $\omega(x) = \lambda_2 \bigl(\|x\|_p + \frac{1}{2}\|x\|_2^2\bigr)$. The subproblems in \eqref{eq:comp-alg-update} are solved using CVX with the Gurobi backend. 

\paragraph{Baselines.}
We compare the performance of our algorithm with two standard approaches: the Lasso model \citep{10.1111/j.2517-6161.1996.tb02080.x},
\begin{align*}
    \min_{x\in\rr^n} \frac{1}{2}\|Ax-b\|_2^2 + \lambda_1\|x\|_1,
\end{align*}
and the Basis Pursuit (BP) model \citep{chen2001atomic},
\begin{align*}
\min_{x\in\R^n}\|x\|_1\quad\text{s.t.}\quad Ax=b.
\end{align*}

\paragraph{Experimental Setup.}
To evaluate the signal recovery performance, we generate synthetic data through the following procedure. The sensing matrix $A \in \rr^{m\times n}$ is drawn from a standard Gaussian distribution with $m=32$ and $n=256$. The true sparse signal $x^* \in \rr^n$ is generated with a sparsity level $\|x^*\|_0 \in \{8,16,24,32\}$, where the nonzero entries are sampled from a standard normal distribution. The observation vector is computed directly as $b = Ax^*$. 

We test the regularization parameters $\lambda_1 = \lambda_2 \in \{0.1, 0.01, 0.001\}$ and vary $p$ across the set $\{1, 1.1, 1.2, \dots, 2\}$ to assess how the $\ell_1-\ell_p$ penalty influences recovery accuracy. The initial point $x_0$ is generated by solving the standard Lasso problem using the built-in MATLAB function \texttt{lasso(A,b,'lambda',lambda)}. The algorithm terminates when the iterates satisfy $\|x_{\nu+1}-x_\nu\|_2/\|x_\nu\|\leq 10^{-6}$ for any $\nu$, or when it reaches a maximum of $1000$ iterations. We repeat each experimental configuration 5 times and report the average and standard deviation of the relative error of the final output $x_{\text{final}}$, defined as:
\begin{align*}
    \text{Relative Error} = \frac{\|x_{\text{final}} - x^*\|_2}{\|x^*\|_2}. 
\end{align*}

\paragraph{Results.}
\begin{figure}[ht]
    \centering
    \includegraphics[width=\linewidth]{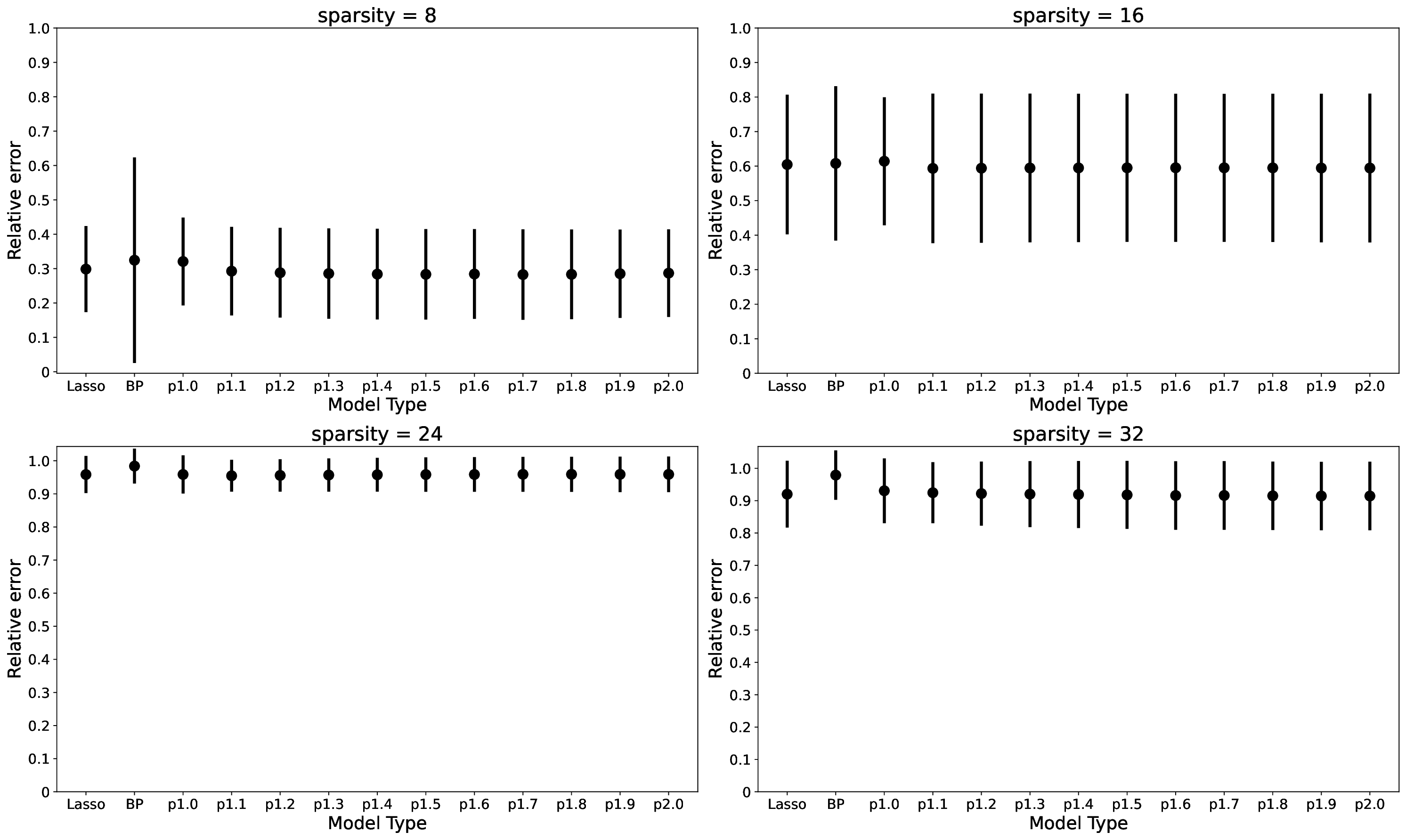}
    \caption{Performance of the Bregman Proximal Subgradient Method on the Compressed Sensing Problem.}
    \label{fig:cs_m32}
\end{figure}

We show the results in Figure~\ref{fig:cs_m32}. For the lowest sparsity level $\|x^*\|_0=8$, the $\ell_1-\ell_p$ regularized models solved by BPG achieve a slightly smaller average relative error than Lasso and BP for most choices of $p>1$. A similar but weaker improvement is observed when $\|x^*\|_0=16$. However, the error bars overlap across methods, so the improvement should be interpreted as moderate rather than decisive. 

When the sparsity level increases to $\|x^*\|_0=24$ and $\|x^*\|_0=32$, all methods exhibit a large relative error (close to one in most cases). This indicates that the recovery problem becomes too difficult under the current sampling configuration, and the advantage of the nonconvex $\ell_1-\ell_p$ regularization is no longer clear. Across the tested values of $p$, the performance is relatively stable; changing $p$ has a visible effect in the low-sparsity case but only a minor effect in the more difficult regimes.

\subsection{Principal Component Analysis}\label{subsec:pca}

In this subsection, we apply the SBMM method to the robust $L_p$ Principal Component Analysis (PCA) problem~\citep{kwak2013principal}:
\begin{align}\label{eq:pca-lp}
\max_{W\in\R^{d\times m}}\frac{1}{n}\sum_{i=1}^{n}\|W^\top x_i\|_p^p \quad \text{s.t.} \quad W^\top W=I_m,
\end{align}
where $X=[x_1,\ldots,x_n]\in\R^{d\times n}$ is the centered data matrix and $W\in\R^{d\times m}$ has orthonormal columns. Equivalently, up to a positive scaling, we minimize
\begin{align}\label{eq:pca-lp-min}
\min_{W\in\R^{d\times m}}f(W)=-\frac1p\sum_{i=1}^{n}\sum_{j=1}^{m}|w_j^\top x_i|^p \quad \text{s.t.} \quad W^\top W=I_m.
\end{align}
When $p=2$, this reduces to standard PCA. For $1<p<2$, the objective is continuously differentiable, but its gradient is not Lipschitz near points where $w_j^\top x_i=0$, while the Stiefel constraint makes the problem nonconvex. We use the reference function
\begin{align*}
\omega(W):=\frac1p\sum_{i=1}^{n}\sum_{j=1}^{m}|w_j^\top x_i|^p.
\end{align*}

\paragraph{Subproblem Formulation.}
In this experiment, we take the full dataset as the minibatch, so the stochastic model and the sampled Bregman distance reduce to their full-batch counterparts. At each iteration $k$, we compute the full gradient
$g_k=\nabla f(W_k)$ using all $n$ samples and solve
\begin{align}
\label{eq:pca-subprob}
V^*=\argmin_V\left\{\langle g_k,V\rangle+\frac1{t_k}D_\omega(W_k+V,W_k)\right\}
\quad \text{s.t.} \quad W_k^\top V+V^\top W_k=0,
\end{align}
where $t_k=\tau_0/\sqrt{k}$ for $k=1,2,\ldots$. The subsequent polar retraction $W_{k+1}=\operatorname{Retr}(W_k,V^*)$ preserves the Stiefel constraint. We implement the SBMM update on the Stiefel manifold through a tangent-space subproblem followed by a polar retraction. We formulate \eqref{eq:pca-subprob} as a conic program using the power-cone constraints $(t_{ij},1,z_{ij})\in\mathcal K_{\mathrm{pow}}(1/p)$, where $z_{ij}=(w_j+v_j)^\top x_i$, and solve it using SCS~\citep{ocpb:16}.

\paragraph{Baselines.}
We compare SBMM with the Non-Greedy PCA-$L_p$ method~\citep{kwak2013principal}, which solves~\eqref{eq:pca-lp} using closed-form SVD updates derived from the full gradient $$\nabla_W F_p = \sum_{i=1}^n x_i\left(\mathrm{sign}(W^\top x_i) \odot |W^\top x_i|^{p-1}\right)^\top.$$ We also test a hybrid strategy (SBMM $+$ Non-Greedy) that runs SBMM first, then uses its output to warm-start the Non-Greedy method.

\paragraph{Experimental Setup.}
We generate synthetic data matrices $X = A_1 A_2 A_3$, where $A_1 \in \rr^{d \times d}$, $A_2 \in \rr^{d \times n}$, and $A_3 \in \rr^{n \times n}$ have i.i.d.\ standard Gaussian entries. The rows of $X$ are centered and normalized. To test sensitivity to initialization, we use two random orthonormal initializations and three additional initializations selected from random candidates with relatively high initial objective values.

We test $p \in \{1.1, 1.2, 1.3\}$ across dimensions $(d, n, m) \in \{(5,10,3), (5,50,3), (5,100,3), (10,100,5)\}$, reporting averages over 5 initializations. SBMM runs for $K = 20{,}000$ iterations, and we record the best objective value attained along the trajectory. The initial step size $\tau_0 \in \{1, 2, 4\}$ is tuned for the best average objective. For the hybrid method, SBMM is run for $200$ iterations before switching to the Non-Greedy method, which is run for at most $400$ iterations.

\paragraph{Results.}
Table~\ref{tab:pca-results} reports the objective values and the average iteration counts. 

\begin{table}[h!]
\centering
\caption{Comparison of SBMM, Non-Greedy PCA-$L_p$, and SBMM $+$ Non-Greedy on the PCA-$L_p$ problem~\eqref{eq:pca-lp-min}. Objective values are averaged over 5 initializations. ``Obj'' reports the objective value $f(W)$; ``Iter'' reports the average number of iterations. For SBMM, we report the best objective attained over the trajectory. For the hybrid method, ``Iter'' reports only the subsequent Non-Greedy iterations.}
\label{tab:pca-results}
\small
\begin{tabular}{cc|rr|rr|rr}
\toprule
 & & \multicolumn{2}{c|}{SBMM} & \multicolumn{2}{c|}{Non-Greedy} & \multicolumn{2}{c}{SBMM + Non-Greedy} \\
$p$ & $(d, n, m)$ & Obj & Iter & Obj & Iter & Obj & Iter \\
\midrule
\multirow{4}{*}{1.1}
 & $(5,\; 10,\; 3)$     & $\mathbf{-28.79}$    & 20000 & $-28.21$    & 27  & $-28.78$    & 16  \\
 & $(5,\; 50,\; 3)$     & $-136.88$   & 20000 & $-136.74$   & 37  & $\mathbf{-137.12}$   & 33  \\
 & $(5,\; 100,\; 3)$    & $-287.59$   & 20000 & $-288.57$   & 30  & $\mathbf{-288.78}$   & 43  \\
 & $(10,\; 100,\; 5)$   & $\mathbf{-525.61}$   & 20000 & $-523.28$   & 45  & $-524.51$   & 45  \\
\midrule
\multirow{4}{*}{1.2}
 & $(5,\; 10,\; 3)$     & $\mathbf{-27.09}$    & 20000 & $-26.85$    & 31  & $-27.02$    & 22  \\
 & $(5,\; 50,\; 3)$     & $-129.09$   & 20000 & $\mathbf{-129.72}$   & 67  & $-129.62$   & 56  \\
 & $(5,\; 100,\; 3)$    & $-270.07$   & 20000 & $-272.66$   & 57  & $\mathbf{-275.16}$   & 63  \\
 & $(10,\; 100,\; 5)$   & $-499.10$   & 20000 & $-500.59$   & 84  & $\mathbf{-500.94}$   & 59  \\
\midrule
\multirow{4}{*}{1.3}
 & $(5,\; 10,\; 3)$     & $\mathbf{-25.75}$    & 20000 & $-25.55$    & 49  & $-25.66$    & 35  \\
 & $(5,\; 50,\; 3)$     & $-123.84$   & 20000 & $\mathbf{-124.60}$   & 75  & $-124.36$   & 83  \\
 & $(5,\; 100,\; 3)$    & $-259.32$   & 20000 & $-262.30$   & 93  & $\mathbf{-263.06}$   & 100 \\
 & $(10,\; 100,\; 5)$   & $-478.22$   & 20000 & $-480.87$   & 123 & $\mathbf{-481.96}$   & 130 \\
\bottomrule
\end{tabular}
\end{table}

The results show that for the smallest setting $(d,n,m)=(5,10,3)$, SBMM attains a lower objective value than the Non-Greedy baseline for all three tested values of $p$. Across all 12 settings, the hybrid method attains the lowest objective value in six cases, SBMM in four cases, and the Non-Greedy method in two cases. In particular, for $p=1.2$ and $(d,n,m)=(5,100,3)$, the hybrid 
method reaches $-275.16$, compared with $-272.66$ for the Non-Greedy method. These results indicate that the SBMM phase can provide a useful warm start in several instances.

\section{Conclusion}

This paper studied Bregman proximal methods for composite optimization problems where both the loss term and the regularizer may be nonsmooth and nonconvex. Under relative weak convexity, we proved Bregman stationarity guarantees for three settings: a deterministic proximal subgradient method, a stochastic model-based method with exact Bregman distance, and a minibatch method where the Bregman distance itself is sampled. In each case, the analysis gives an $\mathcal O(\varepsilon^{-4})$ iteration complexity for finding an $\varepsilon$-stationary point under the corresponding subgradient, model accuracy, and stochastic geometry assumptions. The numerical experiments on compressed sensing and robust $L_p$ PCA illustrate that Bregman geometry can be useful for structured nonconvex problems, although the empirical gains depend on the problem regime and on how the proximal subproblems are solved.

\newpage
\bibliographystyle{abbrvnat} 
\bibliography{reference}

\appendix

\section{Properties of Bregman distance and Bregman Proximal Gradient Methods}
By the definition of Bregman distance, we have the following lower bound for Bregman distance.
\begin{lemma} \label{lem:bregman_strongly_convex_lower_bound}
	Let $\omega:\R^d\to\R$ be a continuously differentiable and $1$-strongly convex function. For any $x,y\in\R^d$, the Bregman distance associated with $\omega$ satisfies
	\begin{align*}
		D_\omega(x,y) \geq \dfrac{1}{2}\|x-y\|^2.
	\end{align*}
\end{lemma}
We also have the following three points identity that will be used in the convergence analysis:
\begin{lemma}[Three points identity] \cite[Lemma 3.1]{chen1993convergence}\label{lem:bregman_three_points_identity}
 	Let $\omega:\R^d\to\R$ be a continuously differentiable and $1$-strongly convex function. For any $x,y,z\in\R^d$, we have
	\begin{align}\label{eq:bregman_three_points_identity}
		[\nabla \omega(y) - \nabla \omega(z)]^\top (x-z) = D_\omega(x,z) + D_\omega(z, y) - D_\omega(x, y)
	\end{align}
\end{lemma}

\begin{lemma}\label{lem:rwc-fb}
    Let $\varphi:\rr^d\to \rr$ be a $\rho\omega$-relatively weakly convex function and  $\lambda\in (0, \frac1\rho)$ be a real number.
	For any $x,z\in\rr^d$, it holds that
    \begin{align*}
        \varphi(z) - \varphi(\xhat)  \geq \frac1\lambda [D_\omega(\xhat,x) - D_\omega(z,x)],
    \end{align*}
	with $\xhat\triangleq \bproxvp(x)$.
\end{lemma}

\begin{proof}
    By the definition of $\bproxvp(\cdot)$, we have
    \begin{align*}
        \varphi(z)+ \frac1\lambda D_\omega(z, x) & \geq \min_y [\varphi(y) + \frac1\lambda D_\omega(y,x)] = \varphi(\bproxvp(x))+ \frac1\lambda D_\omega(\bproxvp(x),x).
    \end{align*}
    The desired inequality follows.
\end{proof}

\begin{lemma}\label{lem:rwc-breg-me}
    Let $\varphi:\rr^d\to \rr$ be a $\rho\omega$-relatively weakly convex function and the real number $\lambda\in (0, \dfrac1\rho)$. For any $x, z\in\rr^d$, it holds that
    \begin{align*}
        \benvvp(z) - \benvvp(x) \leq \frac1\lambda [D_\omega(\xhat,z) - D_\omega(\xhat,x)],
    \end{align*} 
	with $\xhat\triangleq \bproxvp(x)$.
\end{lemma}

\begin{proof}
    By the definitions of $\benvvp$ and $\bproxvp$, we have
    \begin{align*}
        \benvvp(z) 
        &= \min_y [\varphi(y) + \frac1\lambda D_\omega(y,z)]\leq \varphi(\bproxvp(x)) + \frac1\lambda D_\omega(\bproxvp(x),z) \\
        &= \benvvp(x) - \frac1\lambda D_\omega(\bproxvp(x),x) + \frac1\lambda D_\omega(\bproxvp(x),z),
    \end{align*}
    which is the desired inequality. 
\end{proof}

\section{Proof of Results in Section \ref{sec:BPG}}\label{sec:pf_2}

We start with the following well-known three point inequality in Lemma \ref{lem:bregman-cvx-triple} of Bregman Moreau Envelope.
\begin{lemma}\cite[Lemma 3.2]{chen1993convergence}\label{lem:bregman-cvx-triple}
	Let $\varphi:\rr^d\to \rr$ be a convex function, $\omega:\rr^d\to\rr$ be a convex and differentiable DGF, and $\lambda>0$. For any $x,z\in\rr^d$, it holds that
	\begin{align*}
		\varphi(z) + \frac1\lambda D_\omega(z,x) \geq \varphi(\xhat) + \frac1\lambda D_\omega(\xhat,x)+ \frac1\lambda D_\omega(z,\xhat),
	\end{align*}
	with $\xhat:= \bproxvp(x)$.
\end{lemma}
From Lemma \ref{lem:bregman-cvx-triple}, we have a straightforward extension from convex to relatively weakly convex function, and we present it in the following Lemma \ref{lem:bregman-rcvx-sym}.
\begin{lemma}
\label{lem:bregman-rcvx-sym}
Let $\varphi:\R^d\to\R\cup\{+\infty\}$ be proper, lower
semicontinuous, and $\rho\omega$-relatively weakly convex. Let
$\lambda\in(0,1/\rho)$, then for every $z\in\R^d$,
\begin{align}
\label{eq:bregman-prox-growth}
    \varphi(z)+\frac{1}{\lambda}D_\omega(z,x)
    \geq
    \varphi(\hat x)+\frac{1}{\lambda}D_\omega(\hat x,x)
    +\left(\frac{1}{\lambda}-\rho\right)
    D_\omega(z,\hat x).
\end{align}
In particular,
\begin{align}
\label{eq:bregman-prox-gap-stationarity}
\varphi(x)-\varphi(\hat x)-\rho D_\omega(\hat x,x)
\geq
\left(\frac{1}{\lambda}-\rho\right)
\left(
D_\omega(\hat x,x)+D_\omega(x,\hat x)
\right) = \lambda(1-\lambda\rho)
    \dsymomelam(\hat x,x).
\end{align}
\end{lemma}

\begin{proof}
Since $\varphi+\rho\omega$ is convex and $\hat x$ minimizes
$\varphi(\cdot)+\lambda^{-1}D_\omega(\cdot,x)$, there exists
$s\in\partial(\varphi+\rho\omega)(\hat x)$ such that
\begin{align}
    s
    +\left(\frac{1}{\lambda}-\rho\right)\nabla\omega(\hat x)
    -\frac{1}{\lambda}\nabla\omega(x)
    =0.
\end{align}
By the convexity of $\varphi+\rho\omega$,
\begin{align}
    \varphi(z)+\rho\omega(z)
    \geq
    \varphi(\hat x)+\rho\omega(\hat x)
    +\langle s,z-\hat x\rangle.
\end{align}
Substituting the expression for $s$ and rearranging gives
\begin{align}
    \varphi(z)+\frac{1}{\lambda}D_\omega(z,x)
    \geq
    \varphi(\hat x)+\frac{1}{\lambda}D_\omega(\hat x,x)
    +\left(\frac{1}{\lambda}-\rho\right)
    D_\omega(z,\hat x),
\end{align}
which proves \eqref{eq:bregman-prox-growth}. Setting $z=x$ and using the definition of $\dsymomelam$ yields
\eqref{eq:bregman-prox-gap-stationarity}. 
\end{proof}

Then we present the following technical lemma for analyzing the one-step descent progress.
\begin{lemma}\label{lem:comp-lem-main}
	Let  $\lambda\in (0, \dfrac1{\rho_F})$ and the step size $t_k\in(0, \lambda]$. Then the iterates $\{x^k\}_{k\geq 0}$ generated by 
    \Cref{alg:comp_bpg} satisfy that
	\begin{align}
			& \frac{1}{t_k}\left(D_\omega(\xhk,\xkp) + D_\omega(\xkp, \xk) - D_\omega(\xhk,\xk)\right) \nonumber\\
			\leq &\  F(\xhk) -  F(\xk) 
			+ \la g_f(\xk), \xk - \xkp\ra + r(\xk)-r(\xkp) \nonumber\\
			& + \rho_r D_\omega(\xhk, \xkp) + \rho_f D_\omega(\xhk, \xk)\label{eq:comp-lem1-way1-ineq-1}
	\end{align}
	and
	\begin{align}
			& \frac{1}{t_k}\left(D_\omega(\xhk,\xkp) + D_\omega(\xkp, \xk) - D_\omega(\xhk,\xk)\right) \nonumber\\
			\leq &\  F(\xhk) -  F(\xkp) 
			+ \la g_f(\xk), \xk - \xkp\ra + f(\xkp)-f(\xk) \nonumber\\
			& + \rho_r D_\omega(\xhk, \xkp) + \rho_f D_\omega(\xhk, \xk)\label{eq:comp-lem1-way2-ineq-1}
	\end{align}
	with $g_f(\xk)\in\partial f(\xk)$.
\end{lemma}

\begin{proof}
	From the optimality condition of \eqref{eq:comp-alg-update}, we have
	\begin{align}
			\frac{1}{t_k}(\nabla \omega(\xk) - \nabla \omega(\xkp)) - g_f(\xk) \in \partial r(\xkp),
	\end{align}
	with $g_f(\xk)\in\partial f(\xk)$.
	From the $\rho_r\omega$-RWC of $r$ , we obtain that for any $x\in\rr^d$,
	\begin{align*}
			r(x) \geq r(\xkp) + \la \frac{1}{t_k}(\nabla \omega(\xk) - \nabla \omega(\xkp)) - g_f(\xk), x - \xkp\ra - \rho_r D_\omega(x, \xkp).
	\end{align*}
	Rearranging the terms yields that
	\begin{align}\label{eq:comp-lem1-ineq-1}
		\begin{split}
			& r(x) - r(\xkp) + \rho_r D_\omega(x, \xkp) + \la g_f(\xk), x - \xkp\ra \\
			\geq &\ \frac{1}{t_k}\la \nabla \omega(\xk) - \nabla \omega(\xkp), x - \xkp\ra \\
			= &\  \frac{1}{t_k}\(D_\omega(x,\xkp) + D_\omega(\xkp, \xk) - D_\omega(x,\xk)\).
		\end{split}
	\end{align}
        where the equality comes from the three points identity \eqref{eq:bregman_three_points_identity}. Since $f$ is $\rho_f\omega$-RWC, we similarly have
	\begin{align}\label{eq:comp-lem1-frwc}
		f(x)- f(\xk) + \rho_f D_\omega(x, \xk) \geq \la g_f(\xk), x - \xk\ra.
	\end{align}
	So combine \eqref{eq:comp-lem1-ineq-1} and \eqref{eq:comp-lem1-frwc}, and we obtain 
	\begin{align*}
        &f(x)- f(\xk) + r(x) - r(\xkp) + \rho_f D_\omega(x, \xk) + \rho_r D_\omega(x, \xkp) + \la g_f(\xk), \xk - \xkp\ra \\
        \ge\ & \frac{1}{t_k}\(D_\omega(x,\xkp) + D_\omega(\xkp, \xk) - D_\omega(x,\xk)\)
	\end{align*}
	Let $\lambda\in (0, \dfrac1{\rho_F})$, and  $x=\xhk\triangleq \prox_{\lambda,F}^\omega(\xk)$. Then rearrange the terms in two different ways and we obtain that
	\begin{align*}
		\begin{split}
			& \frac{1}{t_k}\(D_\omega(\xhk,\xkp) + D_\omega(\xkp, \xk) - D_\omega(\xhk,\xk)\) \\
			\leq &\  F(\xhk) -  F(\xk) 
			+ \la g_f(\xk), \xk - \xkp\ra + r(\xk)-r(\xkp) \\
			& + \rho_r D_\omega(\xhk, \xkp) + \rho_f D_\omega(\xhk, \xk),
		\end{split}
	\end{align*}
	and
	\begin{align*}
		\begin{split}
			& \frac{1}{t_k}\(D_\omega(\xhk,\xkp) + D_\omega(\xkp, \xk) - D_\omega(\xhk,\xk)\) \\
			\leq &\  F(\xhk) -  F(\xkp) 
			+ \la g_f(\xk), \xk - \xkp\ra + f(\xkp)-f(\xk) \\
			& + \rho_r D_\omega(\xhk, \xkp) + \rho_f D_\omega(\xhk, \xk).
		\end{split}
	\end{align*}
\end{proof}
\subsection{Proof of \Cref{lem:comp-lem-rec-gf}}\label{sec:pf_of_2lem_1}

\begin{proof}
	From the optimality condition of \eqref{eq:comp-alg-update} and the $\rho_r\omega$-RWC of $r$, we have
	\begin{align*}
		&(1-t_k\rho_r)D_\omega(\xhk,\xkp)
		+D_\omega(\xkp,\xk)
		-D_\omega(\xhk,\xk)\\
		\leq&\
		t_k\left[
		\la g_k,\xhk-\xkp\ra
		+r(\xhk)-r(\xkp)
		\right].
	\end{align*}
	Since $f$ is $\rho_f\omega$-RWC, we have
	\begin{align*}
		\la g_k,\xhk-\xk\ra
		\leq
		f(\xhk)-f(\xk)
		+\rho_fD_\omega(\xhk,\xk).
	\end{align*}
	Combining the above two inequalities, using $\rho_F=\rho_f+\rho_r$, and rearranging the terms, we obtain
	\begin{align*}
		&(1-t_k\rho_r)
		\left(
		D_\omega(\xhk,\xkp)
		-D_\omega(\xhk,\xk)
		\right)
		+D_\omega(\xkp,\xk)\\
		\leq&\
		t_k\left[
		F(\xk)-F(\xkp)
		-\left(
		F(\xk)-F(\xhk)
		-\rho_FD_\omega(\xhk,\xk)
		\right)
		\right]\\
		&+
		t_k\left[
		\la g_k,\xk-\xkp\ra
		+f(\xkp)-f(\xk)
		\right].
	\end{align*}
	Since $\xhk$ is a feasible point in the definition of $\benvF(\xkp)$, we have
	\begin{align*}
		\benvF(\xkp)-\benvF(\xk)
		\leq
		\frac{1}{\lambda}
		\left(
		D_\omega(\xhk,\xkp)
		-D_\omega(\xhk,\xk)
		\right).
	\end{align*}
	Combining the above two inequalities yields
	\begin{align*}
		&\frac{t_k}{\lambda(1-t_k\rho_r)}
		\left(
		F(\xk)-F(\xhk)
		-\rho_FD_\omega(\xhk,\xk)
		\right)\\
		\leq&\
		\benvF(\xk)-\benvF(\xkp)
		+
		\frac{t_k}{\lambda(1-t_k\rho_r)}
		\left(
		F(\xk)-F(\xkp)
		\right)\\
		&+
		\frac{1}{\lambda(1-t_k\rho_r)}
		\left[
		t_k\left(
		\la g_k,\xk-\xkp\ra
		+f(\xkp)-f(\xk)
		\right)
		-D_\omega(\xkp,\xk)
		\right].
	\end{align*}
	Since $f$ is $L_f$-Lipschitz and $\|g_k\|\leq L_f$, it follows from the $1$-strong convexity of $\omega$ that
	\begin{align*}
		&t_k\left(
		\la g_k,\xk-\xkp\ra
		+f(\xkp)-f(\xk)
		\right)
		-D_\omega(\xkp,\xk)\\
		\leq&\
		2t_kL_f\|\xkp-\xk\|
		-\frac{1}{2}\|\xkp-\xk\|^2
		\leq2t_k^2L_f^2.
	\end{align*}
	Finally, applying \eqref{eq:bregman-prox-gap-stationarity} to $F$, we have
	\begin{align*}
		F(\xk)-F(\xhk)
		-\rho_FD_\omega(\xhk,\xk)
		\geq
		\lambda(1-\lambda\rho_F)
		\dsymomelam(\xhk,\xk).
	\end{align*}
	Substituting the above two inequalities into the preceding inequality, we obtain
	\begin{align*}
		\frac{t_k(1-\lambda\rho_F)}
		{1-t_k\rho_r}
		\dsymomelam(\xhk,\xk)
		\leq
		\benvF(\xk)-\benvF(\xkp)+
		\frac{t_k}{\lambda(1-t_k\rho_r)}
		\left(
		F(\xk)-F(\xkp)
		\right)+\frac{2t_k^2L_f^2}
		{\lambda(1-t_k\rho_r)},
	\end{align*}
	which proves \eqref{eq:comp-lem-gf}.
\end{proof}

\subsection{Proof of \Cref{lem:comp-lem-rec-gF}}\label{sec:pf_of_2lem_2}
\begin{proof}
	Denote $\mathcal G_k:=F(\xk)-F(\xhk)-\rho_FD_\omega(\xhk,\xk).$
	Choose $g_r^k\in\partial r(\xk)$ and define $g_F(\xk):=g_k+g_r^k.$ By the $\rho_r\omega$-RWC of $r$, we have
	\begin{align}
		r(\xk)-r(\xkp)
		\leq
		\la g_r^k,\xk-\xkp\ra
		+
		\rho_rD_\omega(\xkp,\xk).
		\label{eq:comp-gF-rwc-r}
	\end{align}
	Combining \eqref{eq:comp-lem1-way1-ineq-1}, with $g_f(\xk)=g_k$, and \eqref{eq:comp-gF-rwc-r}, we obtain
	\begin{align*}
		&\frac{1}{t_k}\left(D_\omega(\xhk,\xkp)+D_\omega(\xkp,\xk)-D_\omega(\xhk,\xk)\right)\\
		\leq&\
		F(\xhk)-F(\xk)+\la g_F(\xk),\xk-\xkp\ra + \rho_rD_\omega(\xkp,\xk) + \rho_rD_\omega(\xhk,\xkp) + \rho_fD_\omega(\xhk,\xk)\\=&-\mathcal G_k+\la g_F(\xk),\xk-\xkp\ra+\rho_r\left(D_\omega(\xhk,\xkp)+D_\omega(\xkp,\xk)-D_\omega(\xhk,\xk)\right).
	\end{align*}
	Therefore, we have
	\begin{align}
		&(1-t_k\rho_r)\left(D_\omega(\xhk,\xkp)+D_\omega(\xkp,\xk)-
		D_\omega(\xhk,\xk)\right)
		\leq t_k\left[-\mathcal G_k+\la g_F(\xk),\xk-\xkp\ra\right].\label{eq:comp-gF-main}
	\end{align}
	Invoking \Cref{lem:rwc-breg-me}, we have
	\begin{align*}
		\benvF(\xkp)-\benvF(\xk)
		\leq\frac{1}{\lambda}\left(D_\omega(\xhk,\xkp)-D_\omega(\xhk,\xk)\right).
	\end{align*}
	It follows from \eqref{eq:comp-gF-main} that
	\begin{align}
		\benvF(\xkp)-\benvF(\xk)\leq-\frac{t_k}{\lambda(1-t_k\rho_r)}\mathcal G_k+\frac{1}{\lambda}\left[\frac{t_k}{1-t_k\rho_r}\la g_F(\xk),\xk-\xkp\ra-D_\omega(\xkp,\xk)\right].
		\label{eq:comp-gF-envelope}
	\end{align}
	By the Cauchy-Schwarz inequality and \Cref{lem:bregman_strongly_convex_lower_bound}, we obtain
	\begin{align}
		&\frac{t_k}{1-t_k\rho_r}
		\la g_F(\xk),\xk-\xkp\ra
		-
		D_\omega(\xkp,\xk)\nonumber\\
		\leq&\
		\frac{t_k}{1-t_k\rho_r}
		\|g_F(\xk)\|
		\|\xk-\xkp\|-\frac{1}{2}\|\xk-\xkp\|^2
		\leq\frac{t_k^2}{2(1-t_k\rho_r)^2}\|g_F(\xk)\|^2.
		\label{eq:comp-gF-young}
	\end{align}
	Finally, applying \eqref{eq:bregman-prox-gap-stationarity} to $F$, we have
	\begin{align}
		\mathcal G_k\geq\left(\frac{1}{\lambda}-\rho_F\right) \left(D_\omega(\xhk,\xk)+D_\omega(\xk,\xhk)\right).\label{eq:comp-gF-gap}
	\end{align}
	Combining \eqref{eq:comp-gF-envelope}, \eqref{eq:comp-gF-young}, and \eqref{eq:comp-gF-gap}, we obtain
	\begin{align*}
		\benvF(\xkp)-\benvF(\xk)
		\leq-\frac{t_k(1-\lambda\rho_F)}
		{(1-t_k\rho_r)\lambda^2}
		\left(D_\omega(\xhk,\xk)+D_\omega(\xk,\xhk)\right)+\frac{t_k^2}
		{2\lambda(1-t_k\rho_r)^2}
		\|g_F(\xk)\|^2,
	\end{align*}
	which proves \eqref{eq:comp-lem-gF}.
\end{proof}

\subsection{Proof of \Cref{2thm:main_cov}}\label{sec:pf_of_2thm_main}
\begin{proof}
We will utilize the one-step descent results of \Cref{lem:comp-lem-rec-gF} and \Cref{lem:comp-lem-rec-gf} to prove the main \Cref{2thm:main_cov}. The entire proof is divided into three parts.
\begin{enumerate}[label = (\arabic*)]

\item Summing up \eqref{eq:comp-lem-gF} with respect to $k$ and applying \ref{2assump:comp-ubF}, we have
\begin{align*}
\dsum_{k=0}^{K-1}\dfrac{1-\lambda \rho_F}{1-t_k\rho_r}t_k \dsymomelam(\xhk,\xk) =&\ \dsum_{k=0}^{K-1}\dfrac{(1-\lambda \rho_F)t_k}{(1-t_k\rho_r)\lambda^2}\(D_\omega(\xhk,\xk) + D_\omega(\xk, \xhk)\) \\
\leq&\ \benvF(x^0)-\benvF(x^K) +\dsum_{k=0}^{K-1}\dfrac{L_F^2}{2\lambda(1-t_k\rho_r)^2}t_k^2
\end{align*}
Since $t_k\in(0, \lambda]$, we have that $1\geq1-t_k\rho_r\geq 1-\lambda\rho_r$. It follows that
\begin{align*}
(1-\lambda \rho_F)\dsum_{k=0}^{K-1}t_k \dsymomelam(\xhk,\xk)\leq \benvF(x^0)-\benvF(x^K) +\dfrac{L_F^2}{2\lambda(1-\lambda\rho_r)^2}\dsum_{k=0}^{K-1}t_k^2
\end{align*}
Therefore we have
\begin{align*}
\((1-\lambda\rho_F)\dsum_{k=0}^{K-1}t_k\)\min_{k=0}^{K-1}\dsymomelam(\xhk,\xk) \leq \benvF(x^0)-F^*+\dfrac{L_F^2}{2\lambda(1-\lambda\rho_r)^2}\dsum_{k=0}^{K-1}t_k^2
\end{align*}
Dividing both sides by $(1-\lambda\rho_F)\sum_{k=0}^{K-1}t_k$ yields the desired result in \eqref{eq:comp-estimate-ubF}, i.e. the convergence result listed below under \Cref{2assum:rwc} and \ref{2assump:comp-ubF}.
\begin{align*}
\min_{k=0}^{K-1}\dsymomelam(\xhk, \xk) \leq \dfrac{\benvF(x^0)-F^*+\dfrac{L_F^2}{2\lambda(1-\lambda\rho_r)^2}\dsum_{k=0}^{K-1}t_k^2}{(1-\lambda\rho_F)\dsum_{k=0}^{K-1}t_k}
\end{align*}
Next, let $\lambda = \dfrac1{2\rho_F}$ and we will prove \eqref{eq:comp-complexity-ubF} by analyzing the following two cases:
\begin{itemize}
    \item First, assume that there exists some $\delta\geq F_{1/(2\rho_F)}^\omega(x^0)-F^*$ satisfying $\sqrt{\frac{(1-\rho_r/(2\rho_F))^2\delta}{\rho_FL_F^2K}} \leq\frac{1}{2\rho_F} $, then we have that $t_k = \sqrt{\frac{(1-\rho_r/(2\rho_F))^2\delta}{\rho_FL_F^2K}}$ and then 
\begin{align*}
&\min_{0\le k\le K-1}\dsymomelam(\xhk,\xk)\\ &\leq\ \dfrac{\benvF(x^0)-F^*+\dfrac{L_F^2}{2\lambda(1-\lambda\rho_r)^2}\dsum_{k=0}^{K-1}t_k^2}{(1-\lambda\rho_F)\dsum_{k=0}^{K-1}t_k}\le\ \dfrac{\delta + \dfrac{L_F^2\rho_F}{(1-\lambda\rho_r)^2}Kt_k^2}{\frac12Kt_k}= \dfrac{2\delta}{Kt_k} + 2\dfrac{L_F^2\rho_Ft_k}{(1-\frac{\rho_r}{2\rho_F})^2}\\
&= \ 2\sqrt{\dfrac{2\delta}{Kt_k} \cdot 2\dfrac{L_F^2\rho_Ft_k}{(1-\frac{\rho_r}{2\rho_F})^2}} = \dfrac{4L_F}{1-\frac{\rho_r}{2\rho_F}}\sqrt{\dfrac{\rho_F\delta}{K}} \le\ 8L_F\sqrt{\dfrac{\rho_F\delta}{K}}
\end{align*}
where the second inequality and the third equality comes from the choice of $\delta$ and $t_k$, and the fourth inequality is from $\rho_r \le \rho_F\ \Longrightarrow\ 1-\frac{\rho_r}{2\rho_F} \ge \frac{1}{2}$.\\
\item Otherwise, we have $\frac{1}{2\rho_F}\leq \sqrt{\frac{(1-\rho_r/(2\rho_F))^2\delta}{\rho_FL_F^2K}}$. It follows that $t_k = \frac{1}{2\rho_F}$ and $L_F\leq 2(1-\frac{\rho_r}{2\rho_F})\sqrt{\frac{\rho_F\delta}{K}}.$ Thus,
\begin{align*}
\min_{0\le k\le K-1}\dsymomelam(\xhk,\xk) &\leq\ \dfrac{\benvF(x^0)-F^*+\dfrac{L_F^2}{2\lambda(1-\lambda\rho_r)^2}\dsum_{k=0}^{K-1}t_k^2}{(1-\lambda\rho_F)\dsum_{k=0}^{K-1}t_k}\\
&\le\ \dfrac{\delta + \dfrac{L_F^2\rho_F}{(1-\lambda\rho_r)^2}Kt_k^2}{\frac12Kt_k}= \dfrac{2\delta}{Kt_k} + 2\dfrac{L_F^2\rho_Ft_k}{(1-\frac{\rho_r}{2\rho_F})^2}= \dfrac{4\delta\rho_F}{K} + \dfrac{L_F^2}{(1-\frac{\rho_r}{2\rho_F})^2} \\
&\le \dfrac{4\rho_F\delta}{K} + \dfrac{4\rho_F\delta}{K} = 8\dfrac{\rho_F\delta}{K}
\end{align*}
where the third inequality comes from $\frac{1}{2\rho_F}\leq \sqrt{\frac{(1-\rho_r/(2\rho_F))^2\delta}{\rho_FL_F^2K}} \Longrightarrow \frac{L_F^2}{(1-\rho_r/(2\rho_F))^2}\le \frac{4\rho_F\delta}{K}$
\end{itemize}
Therefore we have \begin{align*}
    \min_{0\le k\le K-1}\dsymphic{1/(2\rho_F)}(\xhk, \xk) \leq 8\max\left\{\dfrac{\rho_F\delta}{K}, L_F\sqrt{\dfrac{\rho_F\delta}{K}}\right\}
\end{align*}

\item Combining \ref{2assump:comp-ubme} and \Cref{lem:comp-lem-rec-gF}, we have 
\begin{align}
\benvF(\xkp)-F^*\leq&\   \benvF(\xk)- F^* - \dfrac{t_k(1-\lambda \rho_F)}{(1-t_k\rho_r)}\dsymomelam(\xhk,\xk) \nonumber\\
&\ + \dfrac{t_k^2}{2(1-t_k\rho_r)^2\lambda}\(\alpha_1 (\benvF(\xk)-F^*)+\beta_1\)\label{eq:comp-lem-rec-Fdiff}
\end{align}
Rearranging the terms and using $t_k\in(0,\lambda]$ yields
\begin{align*}
&\benvF(\xkp)-F^* \\
\leq&\   \(1+\dfrac{t_k^2\alpha_1}{2(1-\lambda\rho_r)^2\lambda}\)(\benvF(\xk)- F^*) - \dfrac{t_k(1-\lambda \rho_F)}{(1-t_k\rho_r)}\dsymomelam(\xhk,\xk)
+ \dfrac{t_k^2\beta_1}{2(1-\lambda\rho_r)^2\lambda}.
\end{align*}
Dropping the nonpositive term, recursively applying the preceding inequality, and using $1+s\leq e^s$, we obtain that, for all $0\leq j\leq K$,
\begin{align*}
\benvF(x^j)-F^*
\leq&\ \left(\benvF(x^0)-F^*+\dfrac{\beta_1}{2(1-\lambda\rho_r)^2\lambda}\dsum_{k=0}^{j-1}t_k^2\right)
 e^{\frac{\alpha_1}{2(1-\lambda\rho_r)^2\lambda}\sum_{k=0}^{j-1}t_k^2} \\
\leq&\ \left(\benvF(x^0)-F^*+\dfrac{C\beta_1}{2(1-\lambda\rho_r)^2\lambda}\right)
 e^{\frac{C\alpha_1}{2(1-\lambda\rho_r)^2\lambda}}
 =C_F.
\end{align*}
Therefore, for all $0\leq k\leq K$ we have $\benvF(\xk) - F^* \leq C_F.$
Then it follows from \eqref{eq:comp-lem-rec-Fdiff} that
\begin{align}
\dsum_{k=0}^{K-1}\dfrac{1-\lambda \rho_F}{1-t_k\rho_r}t_k \dsymomelam(\xhk,\xk)\leq \benvF(x^0)-\benvF(x^K) +\dsum_{k=0}^{K-1}\dfrac{(\alpha_1C_F+\beta_1)}{2\lambda(1-t_k\rho_r)^2}t_k^2 \label{eq:2pf_main_thm_1}
\end{align}
Similar to the proof in part (a), we obtain the desired inequality \eqref{eq:comp-estimate-ubME}.

Next, we will show that \eqref{eq:comp-complexity-ubME} holds. Let $\lambda = \frac{1}{2\rho_F}$ and $C=\Cbar$. Since $Kt_k^2\leq\Cbar,$
we have
\begin{align*}
C_F\leq&\left(\delta+\frac{\Cbar\beta_1\rho_F}{\left(1-\frac{\rho_r}{2\rho_F}\right)^2}\right)
e^{\frac{\Cbar\alpha_1\rho_F}{\left(1-\frac{\rho_r}{2\rho_F}\right)^2}}
\leq\left(\delta+4\Cbar\beta_1\rho_F\right)e^{4\Cbar\alpha_1\rho_F}\triangleq\Cbar_F.
\end{align*}
It follows from \eqref{eq:comp-estimate-ubME} that
\begin{align}
\min_{0\le k\le K-1}\dsymomelam(\xhk,\xk)
\leq\frac{2\delta}{Kt_k}+\frac{2\rho_F(\alpha_1\Cbar_F+\beta_1)}{\left(1-\frac{\rho_r}{2\rho_F}\right)^2}t_k.
\label{eq:comp-ubME-three-cases}
\end{align}
We consider the following three cases.
\begin{itemize}
\item First, suppose that $t_k=\frac{1}{2\rho_F}.$
According to the choice of $t_k$, we have
\begin{align*}
\frac{\alpha_1\Cbar_F+\beta_1}
{\left(1-\frac{\rho_r}{2\rho_F}\right)^2}
\leq
\frac{4\rho_F\delta}{K}.
\end{align*}
It follows from \eqref{eq:comp-ubME-three-cases} that
\begin{align*}
\min_{0\le k\le K-1}\dsymomelam(\xhk,\xk)
\leq
\frac{4\rho_F\delta}{K}
+
\frac{\alpha_1\Cbar_F+\beta_1}
{\left(1-\frac{\rho_r}{2\rho_F}\right)^2}
\leq
\frac{8\rho_F\delta}{K}.
\end{align*}

\item Next, suppose that
\begin{align*}
t_k
=
\sqrt{
\frac{
\left(1-\frac{\rho_r}{2\rho_F}\right)^2\delta
}{
\rho_F(\alpha_1\Cbar_F+\beta_1)K
}
}.
\end{align*}
Then \eqref{eq:comp-ubME-three-cases} gives
\begin{align*}
\min_{0\le k\le K-1}\dsymomelam(\xhk,\xk)
\leq&
\frac{4}
{1-\frac{\rho_r}{2\rho_F}}
\sqrt{
\frac{
(\alpha_1\Cbar_F+\beta_1)\rho_F\delta}{K}}
\leq8\sqrt{\frac{(\alpha_1\Cbar_F+\beta_1)\rho_F\delta}{K}}.
\end{align*}

\item Finally, suppose that $t_k=\sqrt{\dfrac{\Cbar}{K}}.$
According to the choice of $t_k$, we have
\begin{align*}
\frac{\rho_F(\alpha_1\Cbar_F+\beta_1)}
{\left(1-\frac{\rho_r}{2\rho_F}\right)^2}
\leq
\frac{\delta}{\Cbar}.
\end{align*}
It follows from \eqref{eq:comp-ubME-three-cases} that
\begin{align*}
\min_{0\le k\le K-1}\dsymomelam(\xhk,\xk)
\leq
\frac{4\delta}{\sqrt{\Cbar K}}.
\end{align*}
\end{itemize}
Combining the above three cases, we obtain
\begin{align*}
\min_{0\le k\le K-1}\dsymomelam(\xhk,\xk)
\leq
8\max\left\{
\frac{\rho_F\delta}{K},
\sqrt{
\frac{
(\alpha_1\Cbar_F+\beta_1)\rho_F\delta
}{K}
},
\frac{\delta}{2\sqrt{\Cbar K}}
\right\},
\end{align*}
which proves \eqref{eq:comp-complexity-ubME}.

\item By \Cref{lem:comp-lem-rec-gf}, we have
\begin{align}
&\quad\ \sum_{k=0}^{K-1}\frac{t_k(1-\lambda\rho_F)}{1-t_k\rho_r}\dsymomelam(\xhk,\xk)\nonumber\\
&\leq\benvF(x^0)-\benvF(x^K)+\sum_{k=0}^{K-1}\frac{t_k}{\lambda(1-t_k\rho_r)}\bigl(F(x^k)-F(x^{k+1})\bigr)+\sum_{k=0}^{K-1}\frac{2t_k^2L_f^2}{\lambda(1-t_k\rho_r)}.
\label{eq:comp-gf-summed}
\end{align}
Define $a_k:=\dfrac{t_k}{\lambda(1-t_k\rho_r)}.$
Since $\{t_k\}$ is nonincreasing, $\{a_k\}$ is also
nonincreasing. Therefore,
\begin{align}
&\sum_{k=0}^{K-1}
a_k\bigl(F(x^k)-F(x^{k+1})\bigr)
\nonumber\\
&=a_0\bigl(F(x^0)-F^*\bigr)+\sum_{k=1}^{K-1}
(a_k-a_{k-1})\bigl(F(x^k)-F^*\bigr)-a_{K-1}\bigl(F(x^K)-F^*\bigr)
\nonumber\\
&\leq
a_0\bigl(F(x^0)-F^*\bigr).
\end{align}
Using $\benvF(x^K)\geq F^*$ in
\eqref{eq:comp-gf-summed} and dividing by $\displaystyle\sum_{k=0}^{K-1}
\frac{t_k(1-\lambda\rho_F)}{1-t_k\rho_r}$
proves \eqref{eq:comp-estimate-ubf}.

For the constant-step result, let
$\lambda=1/(2\rho_F)$ and $t=
\min\left\{\frac{1}{2\rho_F},\frac{1}{2L_f}\sqrt{\frac{\delta}{\rho_FK}}\right\}.$
The constant-step specialization gives
\begin{align}
\min_{0\leq k\leq K-1}
\dsymphic{1/(2\rho_F)}(\xhk,\xk)
\leq
\frac{2\delta}{Kt}
+
\frac{4\rho_F\delta}{K}
+
8\rho_FtL_f^2.
\end{align}
If $t=1/(2\rho_F)$, then
$L_f^2\leq\rho_F\delta/K$, and the right-hand side is at most
\begin{align}
    \frac{12\rho_F\delta}{K}.
\end{align}
Otherwise,
\begin{align}
t=
\frac{1}{2L_f}
\sqrt{\frac{\delta}{\rho_FK}},
\end{align}
and the right-hand side is at most
\begin{align}
8L_f\sqrt{\frac{\rho_F\delta}{K}}
+
\frac{4\rho_F\delta}{K}.
\end{align}
This proves \eqref{eq:comp-complexity-ubf}.
		
\end{enumerate}
\end{proof}

\section{Proof of Results in Section \ref{sec:BSMM}}
\subsection{Proof of \Cref{lem:model-exp-rwc}}\label{sec:pf_of_3lem_1}
\begin{proof}
    Let $x,y\in\rr^d$ and $\lambda\in[0,1]$. Denote that $\xbar\triangleq \lambda x + (1-\lambda)y$. 
    Then, by \ref{asp:model-unbiased} and the convexity of $\smf{\xbar}(\cdot;\xi) + r(\cdot) + \rho\omega(\cdot)$ from \ref{asp:model-rwc}, we have 
    \begin{align*}
        & f(\xbar) + r(\xbar) + (\tau+\rho)\omega(\xbar) \\
        = &\, \E_{\xi} [\smf{\xbar}(\xbar;\xi) + r(\xbar) + \rho\omega(\xbar)] + \tau\omega(\xbar)\\
        \leq &\, 
        \lambda\E_{\xi} [\smf{\xbar}(x;\xi) + r(x)+\rho \omega(x)] + (1-\lambda)\E_{\xi} [\smf{\xbar}(y;\xi)+ r(y)+\rho \omega(y)]  + \tau\omega(\xbar) \\
        = &\, \lambda \(f(x) +  r(x)+\rho \omega(x)\) + (1-\lambda)\(f(y) + r(y)+\rho \omega(y)\) \\
        & + \lambda\E_{\xi} [\smf{\xbar}(x;\xi) - f(x)] + (1-\lambda)\E_{\xi} [\smf{\xbar}(y;\xi)-f(y)] + \tau\omega(\xbar).
    \end{align*}
    Invoking \ref{asp:model-unbiased} again, we have
    \begin{align*}
        \lambda\E_{\xi} [\smf{\xbar}(x;\xi) - f(x)] + (1-\lambda)\E_{\xi} [\smf{\xbar}(y;\xi)-f(y)] \leq\lambda\tau D_\omega(x, \xbar) + (1-\lambda)\tau D_\omega(y,\xbar).
    \end{align*}
    It follows that
    \begin{align*}
        & f(\xbar) + r(\xbar) + (\tau+\rho)\omega(\xbar) \\
        \leq &\, \lambda \(f(x) +  r(x)+\rho \omega(x)\) + (1-\lambda)\(f(y) + r(y)+\rho \omega(y)\)\\
        & + \lambda\tau D_\omega(x, \xbar) + (1-\lambda)\tau D_\omega(y,\xbar)+ \tau\omega(\xbar) \\
        = &\, \lambda \(f(x) +  r(x)+ (\rho+\tau)\omega(x)\) + (1-\lambda)\(f(y) + r(y)+(\rho+\tau)\omega(y)\),
    \end{align*}
    which shows that $f(\cdot)+r(\cdot) + (\tau+\rho)\omega(\cdot)$ is convex. Therefore, the function $f(\cdot)+r(\cdot)$ is $(\tau+\rho)\omega$-relatively weakly convex.
\end{proof}

\subsection{Proof of \Cref{3thm:main}}\label{sec:pf_of_3thm_main}
\begin{proof}
    Taking full expectation in \eqref{eq:model-recursion}, we obtain
    \begin{align}\label{eq:model-thm-recursion}
    \frac{t_k(1-\lambda(\tau+\rho))}{1-t_k\rho}\E[\dsymomelam(\xhk,\xk)]\leq&\ \E[\benvF(\xk)]-\E[\benvF(\xkp)]+\frac{t_k}{\lambda(1-t_k\rho)}\left(\E[F(\xk)]-\E[F(\xkp)]\right)\nonumber\\
    &+\frac{t_k^2}{2\lambda(1-t_k\rho)(1-t_k\tau)}\left(\E[L^2(\xkp)]+\E[L^2(\xk)]\right).
    \end{align}
    \begin{enumerate}
        \item Under \ref{asp:model-lip-f}, \eqref{eq:model-thm-recursion} gives
        \begin{align*}
        \frac{t_k(1-\lambda(\tau+\rho))}{1-t_k\rho}\E[\dsymomelam(\xhk,\xk)]\leq&\ \E[\benvF(\xk)]-\E[\benvF(\xkp)]+\frac{t_k}{\lambda(1-t_k\rho)}\left(\E[F(\xk)]-\E[F(\xkp)]\right)\\
        &+\frac{t_k^2L_f^2}{\lambda(1-t_k\rho)(1-t_k\tau)}.
        \end{align*}
        Since $\{t_k\}_{k\geq0}$ is nonincreasing, we have
        \begin{align*}
        &\sum_{k=0}^{K-1}\frac{t_k}{\lambda(1-t_k\rho)}\left(\E[F(\xk)]-\E[F(\xkp)]\right)\\
        =&\ \frac{t_0}{\lambda(1-t_0\rho)}\left(F(x^0)-F^*\right)+\sum_{k=1}^{K-1}\left(\frac{t_k}{\lambda(1-t_k\rho)}-\frac{t_{k-1}}{\lambda(1-t_{k-1}\rho)}\right)\left(\E[F(\xk)]-F^*\right)\\
        &-\frac{t_{K-1}}{\lambda(1-t_{K-1}\rho)}\left(\E[F(x^K)]-F^*\right)\\
        \leq&\ \frac{t_0}{\lambda(1-t_0\rho)}\left(F(x^0)-F^*\right).
        \end{align*}
        Summing the one-step inequality with respect to $k$ and using $\E[\benvF(x^K)]\geq F^*$, we obtain
        \begin{align*}
        &(1-\lambda(\tau+\rho))\sum_{k=0}^{K-1}t_k\E[\dsymomelam(\xhk,\xk)]\\
        \leq&\ \benvF(x^0)-F^*+\frac{t_0}{\lambda(1-t_0\rho)}\left(F(x^0)-F^*\right)+\frac{L_f^2}{\lambda}\sum_{k=0}^{K-1}\frac{t_k^2}{(1-t_k\rho)(1-t_k\tau)}.
        \end{align*}
        Applying the sampling rule in \Cref{alg:bsmm} proves \eqref{eq:model-thm-estimate}.

        Next, let $\lambda=1/(2(\tau+\rho))$ and $t_k=c/\sqrt K$. Since $c\leq\min\{1/(2\rho),1/(2\tau)\}$, we have $1-t_k\rho\geq1/2$ and $1-t_k\tau\geq1/2$. Therefore, \eqref{eq:model-thm-estimate} gives
        \begin{align*}
        \E[\dsymphic{1/(2(\tau+\rho))}(\xhat^{\kbar},\xkbar)]\leq\frac{2}{c\sqrt K}\left(\delta+\frac{4(\tau+\rho)c\delta}{\sqrt K}+8(\tau+\rho)L_f^2c^2\right),
        \end{align*}
        which proves \eqref{eq:model-thm-complexity-ubf}.

        \item Under \ref{asp:model-lip-ME}, \eqref{eq:model-thm-recursion} gives
        \begin{align}\label{eq:model-thm-recursion-ME}
        &\quad\ \frac{t_k(1-\lambda(\tau+\rho))}{1-t_k\rho}\E[\dsymomelam(\xhk,\xk)]\nonumber\\
        \leq&\ \E[\benvF(\xk)-F^*]-\E[\benvF(\xkp)-F^*]
        +\frac{t_k}{\lambda(1-t_k\rho)}\left(\E[F(\xk)]-\E[F(\xkp)]\right)\nonumber\\
        &+\frac{\alpha_1t_k^2}{2\lambda(1-t_k\rho)(1-t_k\tau)}\left(\E[\benvF(\xkp)-F^*]+\E[\benvF(\xk)-F^*]\right)+\frac{\beta_1t_k^2}{\lambda(1-t_k\rho)(1-t_k\tau)}.
        \end{align}
        For any $1\leq j\leq K$, summing \eqref{eq:model-thm-recursion-ME} from $k=0$ to $j-1$, dropping the nonnegative stationarity term, and using the same summation-by-parts argument as in part (1), we obtain
        \begin{align*}
        \E[\benvF(x^j)-F^*]\leq&\ \benvF(x^0)-F^*+\frac{t_0}{\lambda(1-t_0\rho)}\left(F(x^0)-F^*\right)\\
        &+\frac{\alpha_1}{\lambda}\sum_{k=0}^{j-1}\frac{t_k^2}{(1-t_k\rho)(1-t_k\tau)}\max_{0\leq i\leq j}\E[\benvF(x^i)-F^*]\\
        &+\frac{\beta_1}{\lambda}\sum_{k=0}^{j-1}\frac{t_k^2}{(1-t_k\rho)(1-t_k\tau)}.
        \end{align*}
        Taking the maximum with respect to $j$, we have
        \begin{align*}
        \max_{0\leq j\leq K}\E[\benvF(x^j)-F^*]\leq\frac{\benvF(x^0)-F^*+\dfrac{t_0}{\lambda(1-t_0\rho)}\left(F(x^0)-F^*\right)+\dfrac{\beta_1}{\lambda}\sum_{k=0}^{K-1}\dfrac{t_k^2}{(1-t_k\rho)(1-t_k\tau)}}{1-\dfrac{\alpha_1}{\lambda}\sum_{k=0}^{K-1}\dfrac{t_k^2}{(1-t_k\rho)(1-t_k\tau)}}.
        \end{align*}
        Summing \eqref{eq:model-thm-recursion-ME} from $k=0$ to $K-1$ and substituting the preceding bound, we obtain
        \begin{align*}
        &(1-\lambda(\tau+\rho))\sum_{k=0}^{K-1}t_k\E[\dsymomelam(\xhk,\xk)]\\
        \leq&\ \frac{\benvF(x^0)-F^*+\dfrac{t_0}{\lambda(1-t_0\rho)}\left(F(x^0)-F^*\right)+\dfrac{\beta_1}{\lambda}\sum_{k=0}^{K-1}\dfrac{t_k^2}{(1-t_k\rho)(1-t_k\tau)}}{1-\dfrac{\alpha_1}{\lambda}\sum_{k=0}^{K-1}\dfrac{t_k^2}{(1-t_k\rho)(1-t_k\tau)}}.
        \end{align*}
        Applying the sampling rule in \Cref{alg:bsmm} proves \eqref{eq:model-thm-estimate-ME}.

        Finally, let $\lambda=1/(2(\tau+\rho))$ and $t_k=c/\sqrt K$. Since $c\leq\min\{1/(2\rho),1/(2\tau)\}$, we have
        \begin{align*}
        \sum_{k=0}^{K-1}\frac{t_k^2}{(1-t_k\rho)(1-t_k\tau)}\leq4c^2.
        \end{align*}
        Therefore, \eqref{eq:model-thm-estimate-ME} gives
        \begin{align*}
        \E[\dsymphic{1/(2(\tau+\rho))}(\xhat^{\kbar},\xkbar)]\leq\frac{2}{c\sqrt K}\frac{\delta+\dfrac{4(\tau+\rho)c\delta}{\sqrt K}+8(\tau+\rho)\beta_1c^2}{1-8\alpha_1(\tau+\rho)c^2},
        \end{align*}
        which proves \eqref{eq:model-thm-complexity-ME}.
    \end{enumerate}
\end{proof}

\subsection{Proof of \Cref{lem:model-recursion}}\label{sec:pf_of_3lem_2}
\begin{proof}
    From the optimality condition of \eqref{eq:3subpro} and the $\rho\omega$-RWC of $\sfxk(\cdot;\xik)+r(\cdot)$, we have
    \begin{align*}
    &\frac1{t_k}\left(D_\omega(\xhk,\xkp)-D_\omega(\xhk,\xk)+D_\omega(\xkp,\xk)\right)\\
    \leq&\ \sfxk(\xhk;\xik)+r(\xhk)-\sfxk(\xkp;\xik)-r(\xkp)+\rho D_\omega(\xhk,\xkp).
    \end{align*}
    Taking conditional expectation and adding and subtracting the true objective, we obtain
    \begin{align*}
    &\frac1{t_k}\E_k\left[D_\omega(\xhk,\xkp)-D_\omega(\xhk,\xk)+D_\omega(\xkp,\xk)\right]\\
    \leq&\ F(\xhk)-\E_k[F(\xkp)]+\E_k[\sfxk(\xhk;\xik)-f(\xhk)]\\
    &+\E_k[f(\xkp)-\sfxk(\xkp;\xik)]+\rho\E_k[D_\omega(\xhk,\xkp)].
    \end{align*}
    By \ref{asp:model-unbiased}, we have
    \begin{align*}
    \E_k[\sfxk(\xhk;\xik)-f(\xhk)]\leq\tau D_\omega(\xhk,\xk).
    \end{align*}
    Moreover, applying \eqref{eq:bregman-prox-gap-stationarity} to $F$, we have
    \begin{align*}
    F(\xhk)-F(\xk)+\tau D_\omega(\xhk,\xk)\leq-\lambda(1-\lambda(\tau+\rho))\dsymomelam(\xhk,\xk)-\rho D_\omega(\xhk,\xk).
    \end{align*}
    Let $\xi\sim P$ be independent of $\{\xi_k\}_{k=0}^{K-1}$. By \ref{asp:model-unbiased} and \ref{asp:model-lip}, we obtain
    \begin{align*}
    &\quad\ \E_k[f(\xkp)-\sfxk(\xkp;\xik)]\nonumber\\
    &\leq\E_k\left[\left(L(\xkp)+\sml(\xk;\xik)\right)\sqrt{\min\{D_\omega(\xkp,\xk),D_\omega(\xk,\xkp)\}}+\tau\min\{D_\omega(\xkp,\xk),D_\omega(\xk,\xkp)\}\right].
    \end{align*}
    Since $t_k\tau<1$, the Cauchy--Schwarz inequality and Young's inequality yield
    \begin{align*}
    t_k\E_k[f(\xkp)-\sfxk(\xkp;\xik)]-\E_k[D_\omega(\xkp,\xk)]
    \leq\frac{t_k^2}{2(1-t_k\tau)}\left(\E_k[L^2(\xkp)]+L^2(\xk)\right).
    \end{align*}
    Combining the above inequalities, we obtain
    \begin{align*}
    &(1-t_k\rho)\E_k\left[D_\omega(\xhk,\xkp)-D_\omega(\xhk,\xk)\right]\\
    \leq&\ -t_k\lambda(1-\lambda(\tau+\rho))\dsymomelam(\xhk,\xk)+t_k\left(F(\xk)-\E_k[F(\xkp)]\right)+\frac{t_k^2}{2(1-t_k\tau)}\left(\E_k[L^2(\xkp)]+L^2(\xk)\right).
    \end{align*}
    Since $\xhk$ is a feasible point in the definition of $\benvF(\xkp)$, we have
    \begin{align*}
    \E_k[\benvF(\xkp)]-\benvF(\xk)\leq\frac1\lambda\E_k\left[D_\omega(\xhk,\xkp)-D_\omega(\xhk,\xk)\right].
    \end{align*}
    Combining the above two inequalities and rearranging the terms proves \eqref{eq:model-recursion}.
\end{proof}

\section{Proof of Results in Section \ref{sec:SBMM}}

\subsection{Proof of \Cref{thm:model-sbmm}}
\begin{proof}
    By \Cref{lem:model-sb-1} and the definition of $\Gamma$, we have
    \begin{align*}
    \E_k[\benvF(x^{k+1})]\leq\benvF(x^k)-\Gamma\dsymomelam(\hat x^k,x^k)+\frac{t^2L^2}{4\lambda c(m)(1-t\rho)}.
    \end{align*}
    Taking full expectation, summing with respect to $k$, and using $\E[\benvF(x^K)]\geq F^*$, we obtain
    \begin{align*}
    \Gamma\sum_{k=0}^{K-1}\E[\dsymomelam(\hat x^k,x^k)]\leq\benvF(x^0)-F^*+\frac{Kt^2L^2}{4\lambda c(m)(1-t\rho)}.
    \end{align*}
    Since $t_k\equiv t$, the sampling rule in \Cref{alg:sbmm} is uniform. Therefore,
    \begin{align*}
    \E[\dsymomelam(\hat x^{\bar k},x^{\bar k})]=\frac1K\sum_{k=0}^{K-1}\E[\dsymomelam(\hat x^k,x^k)].
    \end{align*}
    Dividing the preceding inequality by $\Gamma K$ proves \eqref{eq:sbmm_final_bound_explicit}.

    Next, let $t=\frac{1}{\rho\sqrt K}$. From the condition on $1-c(m)$ in \Cref{thm:model-sbmm}, we have
    \begin{align*}
    \Gamma\geq\frac{t(1-\lambda(\rho+\tau))}{2c(m)(1-t\rho)}.
    \end{align*}
    Substituting this inequality into \eqref{eq:sbmm_final_bound_explicit}, we obtain
    \begin{align*}
    \E[\dsymomelam(\hat x^{\bar k},x^{\bar k})]\leq\frac{2c(m)(1-t\rho)(\benvF(x^0)-F^*)}{Kt(1-\lambda(\rho+\tau))}+\frac{tL^2}{2\lambda(1-\lambda(\rho+\tau))}.
    \end{align*}
    Using $c(m)\leq1$, $1-t\rho\leq1$, and $t=\frac{1}{\rho\sqrt K}$ proves \eqref{eq:sbmm_rate_O1sqrtK_explicit}.
\end{proof}

\subsection{Proof of \Cref{lem:model-sb-1}}
\begin{proof}
At $k$th iteration, by the optimality condition, we obtain that
\begin{align*}
    0\in \partial \sFxk(\xkp; \Bcal_k) +\frac{1}{t_k}\left(\snome(\xkp;\Bcal_k)-\snome(\xk;\Bcal_k)\right).
\end{align*} 
By the $\rho\some(\cdot;\Bcal_k)$-relatively weak convexity of $\sFxk (\cdot; \Bcal_k) = \sfxk(\cdot;\Bcal_k) + r(\cdot)$, we derive that for any $x\in\rr^d$,
\begin{align*}
    \sfxk(x;\Bcal_k) + r(x) 
    \geq&\ \sfxk(\xkp;\Bcal_k) +
    r(\xkp) + \dfrac{1}{t_k}\la \snome(x^k; \Bcal_k) - \snome(\xkp; \Bcal_k), x-\xkp \ra \\
    &-\rho \sdome(x, \xkp; \Bcal_k).
\end{align*} 
By the three point identity, we have that
\begin{align*}
    \sfxk(x;\Bcal_k) + r(x) 
    \geq &\ \sfxk(\xkp;\Bcal_k) + r(\xkp)  \\
    & + \dfrac{1}{t_k}\(\sdome(x, \xkp; \Bcal_k) - \sdome(x, \xk; \Bcal_k)+ \sdome(\xkp, \xk; \Bcal_k)\) 
    -\rho \sdome(x, \xkp; \Bcal_k).
\end{align*} 
Rearranging the terms, we obtain that
\begin{align*}
    & \dfrac{1}{t_k}\(\sdome(x, \xkp;\bcalk) - \sdome(x, \xk;\bcalk) + \sdome(\xkp, \xk;\bcalk)\) \\
    \leq &\, f(x) + r(x) - \(f(\xk) +r(\xk)\) +\sfxk(x;\bcalk) - f(x)\\ 
    &  +f(\xk) + r(\xk)- \(\sfxk(\xkp;\bcalk)+r(\xkp)\) + \rho \sdome(x, \xkp;\bcalk).
\end{align*}
Taking expectations in \ref{asp:model-sb-rwc} and using \ref{asp:model-sb-dgf-unbiased} and \ref{asp:model-sb-f}, we obtain that $F=f+r$ is $(\rho+\tau)\omega$-relatively weakly convex. Thus, for $\lambda\in(0,1/(\rho+\tau))$, we define $\xhk\triangleq\bproxF(x^k)=\argmin_y\left\{F(y)+\frac1\lambda D_\omega(y,\xk)\right\}$. Also, taking expectation with respect to $\bcalk\mid\sigma_k$ from both sides and plugging in $x=\xhk$, we derive that
\begin{align}\label{eq:sbmm-lem1-1}
    &\dfrac{1}{t_k}\E_k\[\sdome(\xhk, \xkp;\bcalk) - \sdome(\xhk, \xk;\bcalk) + \sdome(\xkp, \xk;\bcalk)\] \nonumber\\
    \leq &\ \E_k\[F(\xhk) - F(\xk)\] 
    +\E_k\[\sfxk(\xhk;\bcalk) - f(\xhk)\] \\
    & +\E_k\[f(\xk) + r(\xk)- \(\sfxk(\xkp;\bcalk)+r(\xkp)\)\] + \rho\E_k\[ \sdome(\xhk, \xkp; \bcalk)\]. \nonumber
\end{align}
We divide the right hand side into three parts and analyze them separately. For the first term, applying \eqref{eq:bregman-prox-gap-stationarity} to $F$, we have
\begin{align}\label{eq:sbmm-lem1-2}
\E_k\[F(\xhk)-F(\xk)\]\leq-\frac1\lambda D_\omega(\xhk,\xk)-\left(\frac1\lambda-\rho-\tau\right)D_\omega(\xk,\xhk).
\end{align}
Let $\xi\sim P$ be independent of $\Bcal_0,\ldots,\Bcal_{K-1}$. By \ref{asp:model-sb-f}, we have
\begin{align}\label{eq:sbmm-lem1-3}
    \E_k\[\sfxk(\xhk;\bcalk) - f(\xhk)\] \leq \tau D_\omega(\xhk, \xk)
\end{align}
and
\begin{align}\label{eq:sbmm-lem1-4}
    f(\xkp) = \E_\xi\[\sfxkp(\xkp;\xi)\], \qquad
    f(\xk) =\E_k\[\sfxk(\xk;\bcalk)\]
\end{align}
Using \eqref{eq:sbmm-lem1-4} and \Cref{asp:model-sb-lip}, we obtain
\begin{align}\label{eq:sbmm-lem1-5}
    &\E_k\[f(\xk) + r(\xk)- \(\sfxk(\xkp;\bcalk)+r(\xkp)\)\] \nonumber \\
    = &\, \E_k\[\sfxk(\xk;\bcalk) + r(\xk)- \(\sfxk(\xkp;\bcalk)+r(\xkp)\)\]\nonumber \\  
    = &\, \dfrac{1}{m_k}\dsum_{i=1}^{m_k}\E_k\[\sfxk(\xk;\xiki) + r(\xk)- \(\sfxk(\xkp;\xiki)+r(\xkp)\)\]\nonumber \\  
    \leq &\, \dfrac{1}{m_k}\dsum_{i=1}^{m_k}\E_k\[\sml(\xk;\xiki)\sqrt{\sdome(\xkp, \xk;\xiki)} \]\\
    \leq &\, \dfrac{1}{m_k}\dsum_{i=1}^{m_k}\(\frac{t_k}4\E_k\[ \sml^2(\xk;\xiki)\] + \frac{1}{t_k}\E_k\[\sdome(\xkp, \xk;\xiki)\]\) \nonumber \\
    \leq &\,\frac{t_k}4 L^2 + \frac{1}{t_k} \cdot\dfrac{1}{m_k}\dsum_{i=1}^{m_k}\E_k\[\sdome(\xkp, \xk;\xiki)\]
    = \, \frac{t_k}4 L^2 + \frac{1}{t_k}\E_k\[\sdome(\xkp, \xk;\bcalk)\].\nonumber 
\end{align}
Finally, combining \eqref{eq:sbmm-lem1-1}, \eqref{eq:sbmm-lem1-2}, \eqref{eq:sbmm-lem1-3}, and \eqref{eq:sbmm-lem1-5}, we obtain
\begin{align*}
&\quad\ \frac1{t_k}\E_k\[\sdome(\xhk,\xkp;\bcalk)\]-\frac1{t_k}D_\omega(\xhk,\xk)+\frac1{t_k}\E_k\[\sdome(\xkp,\xk;\bcalk)\]\nonumber\\
\leq&-\frac1\lambda D_\omega(\xhk,\xk)-\left(\frac1\lambda-\rho-\tau\right)D_\omega(\xk,\xhk)\\
&+\tau D_\omega(\xhk,\xk)+\frac{t_k}4L^2+\frac1{t_k}\E_k\[\sdome(\xkp,\xk;\bcalk)\]+\rho\E_k\[\sdome(\xhk,\xkp;\bcalk)\].
\end{align*}
Canceling the common term, adding $\rho D_\omega(\xhk,\xk)$ to both sides, and multiplying by $\frac{t_k}{1-t_k\rho}$, we obtain
\begin{align*}
\E_k\[\sdome(\xhk,\xkp;\bcalk)\]-D_\omega(\xhk,\xk)\leq-\frac{t_k\lambda(1-\lambda(\rho+\tau))}{1-t_k\rho}\dsymomelam(\xhk,\xk)+\frac{t_k^2L^2}{4(1-t_k\rho)}.
\end{align*}
By \Cref{asp:model-sb-sbg}, we have
\begin{align*}
c(m_k)\E_k[D_\omega(\xhk,\xkp)]-D_\omega(\xhk,\xk)\leq-\frac{t_k\lambda(1-\lambda(\rho+\tau))}{1-t_k\rho}\dsymomelam(\xhk,\xk)+\frac{t_k^2L^2}{4(1-t_k\rho)}.
\end{align*}
Therefore,
\begin{align*}
&c(m_k)\left(\E_k[D_\omega(\xhk,\xkp)]-D_\omega(\xhk,\xk)\right)\\
\leq&\ -\frac{t_k\lambda(1-\lambda(\rho+\tau))}{1-t_k\rho}\dsymomelam(\xhk,\xk)+(1-c(m_k))D_\omega(\xhk,\xk)+\frac{t_k^2L^2}{4(1-t_k\rho)}.
\end{align*}
Since $D_\omega(\xhk,\xk)\leq\lambda^2\dsymomelam(\xhk,\xk)$, we obtain
\begin{align*}
&c(m_k)\left(\E_k[D_\omega(\xhk,\xkp)]-D_\omega(\xhk,\xk)\right)\\
\leq&\ -\left(\frac{t_k\lambda(1-\lambda(\rho+\tau))}{1-t_k\rho}-\lambda^2(1-c(m_k))\right)\dsymomelam(\xhk,\xk)+\frac{t_k^2L^2}{4(1-t_k\rho)}.
\end{align*}
Since $\xhk$ is feasible in the definition of $\benvF(\xkp)$, we have
\begin{align*}
\E_k[\benvF(\xkp)]-\benvF(\xk)\leq\frac1\lambda\E_k[D_\omega(\xhk,\xkp)-D_\omega(\xhk,\xk)].
\end{align*}
Combining the preceding two inequalities proves \eqref{eq:lem_model_sb_1}.
\end{proof}

\end{document}